\documentclass[a4paper,11pt]{article}
\usepackage{amsfonts,amssymb,amsmath, dsfont}
\usepackage{bm}
\usepackage{color}
\usepackage{upgreek}
\usepackage[english]{babel}
\usepackage{graphicx}

 \makeatletter
 \@addtoreset{equation}{section}
 \makeatother

 \makeatletter
 \@addtoreset{enunciato}{section}
 \makeatother

 \newcounter{enunciato}[section]

 \newtheorem{ittheorem}{Theorem}
 \newtheorem{itlemma}{Lemma}
 \newtheorem{itproposition}{Proposition}
 \newtheorem{itcorollary}{Corollary}
 \newtheorem{itdefinition}{Definition}
 \newtheorem{itremark}{Remark}
 \newtheorem{itclaim}{Claim}
 \newtheorem{itfact}{Fact}
 \newtheorem{itconjecture}{Conjecture}

 \newenvironment{theorem}{\addtocounter{enunciato}{1}
 \begin{ittheorem}}{\end{ittheorem}}

 \newenvironment{lemma}{\addtocounter{enunciato}{1}
 \begin{itlemma}}{\end{itlemma}}

 \newenvironment{proposition}{\addtocounter{enunciato}{1}
 \begin{itproposition}}{\end{itproposition}}

 \newenvironment{corollary}{\addtocounter{enunciato}{1}
 \begin{itcorollary}}{\end{itcorollary}}

 \newenvironment{definition}{\addtocounter{enunciato}{1}
 \begin{itdefinition}}{\end{itdefinition}}

 \newenvironment{remark}{\addtocounter{enunciato}{1}
 \begin{itremark}}{\end{itremark}}

 \newenvironment{claim}{\addtocounter{enunciato}{1}
 \begin{itclaim}}{\end{itclaim}}

 \newenvironment{fact}{\addtocounter{enunciato}{1}
 \begin{itfact}}{\end{itfact}}

 \newenvironment{conjecture}{\addtocounter{enunciato}{1}
 \begin{itconjecture}}{\end{itconjecture}}

 \newcommand{\be}[1]{\begin{equation}\label{#1}}
 \newcommand{\ee}{\end{equation}}

 \newcommand{\bl}[1]{\begin{lemma}\label{#1}}
 \newcommand{\el}{\end{lemma}}

 \newcommand{\br}[1]{\begin{remark}\label{#1}}
 \newcommand{\er}{\end{remark}}

 \newcommand{\bt}[1]{\begin{theorem}\label{#1}}
 \newcommand{\et}{\end{theorem}}

 \newcommand{\bd}[1]{\begin{definition}\label{#1}}
 \newcommand{\ed}{\end{definition}}

 \newcommand{\bcl}[1]{\begin{claim}\label{#1}}
 \newcommand{\ecl}{\end{claim}}

 \newcommand{\bfact}[1]{\begin{fact}\label{#1}}
 \newcommand{\efact}{\end{fact}}

 \newcommand{\bp}[1]{\begin{proposition}\label{#1}}
 \newcommand{\ep}{\end{proposition}}

 \newcommand{\bc}[1]{\begin{corollary}\label{#1}}
 \newcommand{\ec}{\end{corollary}}

 \newcommand{\bcj}[1]{\begin{conjecture}\label{#1}}
 \newcommand{\ecj}{\end{conjecture}}

 \newcommand{\bpr}{\begin{proof}}
 \newcommand{\epr}{\end{proof}}

 \newcommand{\bprlem}[1]{\begin{proofof}{\it Lemma \ref{#1}}.\,\,}
 \newcommand{\eprlem}{\end{proofof}}

 \newcommand{\bprthm}[1]{\begin{proofof}{\it Theorem \ref{#1}}.\,\,}
 \newcommand{\eprthm}{\end{proofof}}

 \newcommand{\bi}{\begin{itemize}}
 \newcommand{\ei}{\end{itemize}}

 \newcommand{\ben}{\begin{enumerate}}
 \newcommand{\een}{\end{enumerate}}

 \newenvironment{proof}{\noindent {\em Proof}.\,\,}{\hspace*{\fill}$\halmos$\medskip}
 \newenvironment{proofof}{\noindent {\em Proof of\,\,}}{\hspace*{\fill}$\halmos$\medskip}
 \newcommand{\halmos}{\rule{1ex}{1.4ex}}

 \newcommand{\one}{{\mathchoice {1\mskip-4mu\mathrm l}
         {1\mskip-4mu\mathrm l}
         {1\mskip-4.5mu\mathrm l}
         {1\mskip-5mu\mathrm l}}}

\def \C {{\mathbb C}}
\def \E {{\mathbb E}}

\def \N {{\mathbb N}}
\def \P {{\mathbb P}}
\def \Q {{\mathbb Q}}
\def \R {{\mathbb R}}
\def \Z {{\mathbb Z}}

\def \ba {\begin{array}}
\def \ea {\end{array}}

\def \lra {\longrightarrow}

\def \Q {{\mathbb{Q}}}
\def \T {{\mathbb{T}}}
\def \lra {{\leftrightarrow}}

\def\one{\rlap{\mbox{\small\rm 1}}\kern.15em 1}

\begin{document}
\title{Metastable Densities for the Contact Process on Power Law Random Graphs}

\author{Thomas Mountford\textsuperscript{1}, Daniel Valesin\textsuperscript{2,3}
and Qiang Yao\textsuperscript{4}}

\footnotetext[1]{\'Ecole Polytechnique F\'ed\'erale de Lausanne,
D\'epartement de Math\'ematiques,
1015 Lausanne, Switzerland}
\footnotetext[2]{University of British Columbia, Department of Mathematics, V6T1Z2 Vancouver, Canada}
\footnotetext[3]{Research funded by the Post-Doctoral Research Fellowship of the Government of Canada}
\footnotetext[4]{Department of Statistics and Actuarial Science,
East China Normal University,
Shanghai 200241, China}

\date{December 06, 2012}
\maketitle

\begin{abstract}
We consider the contact process on a random graph with fixed degree distribution given by a power law. We follow the work of Chatterjee and Durrett \cite{CD}, who showed that for arbitrarily small infection parameter $\lambda$, the survival time of the process is larger than a stretched exponential function of the number of vertices, $n$. We obtain sharp bounds for the typical density of infected sites in the graph, as $\lambda$ is kept fixed and $n$ tends to infinity. We exhibit three different regimes for this density, depending on the tail of the degree law.\medskip
  contact process, random graphs.
\end{abstract}

{\bf\large{}}\bigskip


\section{Introduction}

\label{Int}

In this paper we study the contact process on a random graph with a fixed degree distribution equal to a power law. Let us briefly describe the contact process and the random graph we consider.

The contact process is an interacting particle system that is commonly taken as a model for the spread of an infection in a population. Given a locally finite graph $G = (V, E)$ and $\lambda > 0$, the contact process on $G$ with infection rate $\lambda$ is a Markov process $(\xi_t)_{t \geq 0}$ with configuration space $\{0,1\}^V$. Vertices of $V$ (also called sites) are interpreted as individuals, which can be either healthy (state 0) or infected (state 1). The infinitesimal generator for the dynamics is
\begin{equation}\Omega f(\xi) = \sum_{x \in V}\left(f(\phi_x\xi) - f(\xi) \right) + \lambda \sum_{\substack{{(x,y):}\\{\{x, y\} \in E}}} \left(f(\phi_{(x,y)}\xi)-f(\xi) \right),\label{eq:gen}\end{equation}
where
$$\phi_x\xi(z) = \left\{\begin{array}{ll}\xi(z),&\text{if } z \neq x;\\0,&\text{if } z = x;\end{array}\right.\qquad \phi_{(x,y)}\xi(z) = \left\{\begin{array}{ll}\xi(z), &\text{if }z\neq y;\\I_{\left\{\max\left(\xi(x),\;\xi(y)\right) = 1\right\}},&\text{if } z = y. \end{array} \right.$$
Here and in the rest of the paper, $I$ denotes the indicator function. Given $A \subset V$, we will write $(\xi^A_t)$ to denote the contact process with the initial configuration $I_A$. If $A = \{x\}$, we write $(\xi_t^x)$. Sometimes we abuse notation and identify the configuration $\xi_t$ with the set $\{x: \xi_t(x) = 1\}$.

We refer the reader to \cite{lig85} and \cite{lig99} for an elementary treatment of the contact process and proofs of the basic properties that we will now review.

The dynamics given by the generator (\ref{eq:gen}) has two forms of transition. First, infected sites become healthy at rate 1; a recovery is then said to have occurred. Second, given an ordered pair of sites $(x, y)$ such that $x$ is infected and $y$ is healthy, $y$ becomes infected at rate $\lambda$; this is called a transmission. 

We note that the configuration in which all individuals are healthy is absorbing for the dynamics. The random time at which this configuration is reached, $\inf\{t: \xi_t = \varnothing\}$ is called the extinction time of the process. A fundamental question for the contact process is: is this time almost surely finite? The answer to this question depends of course on the underlying graph $G$ and on the rate $\lambda$, but not on the initial configuration $\xi_0$, as long as $\xi_0$ contains a finite and non-zero quantity of infected sites. If, for one such $\xi_0$ (and hence all such $\xi_0$), the extinction time is almost surely finite, then the process is said to die out; otherwise it is said to survive. Using the graphical construction described below, it is very simple to verify that on finite graphs, the contact process dies out.

In order to make an analogy with the contact process on the random graphs we are interested in, it will be useful for us to briefly look at known results for the contact process on the $d$-dimensional lattice $\Z^d$ and on finite boxes of $\Z^d$. The contact process on $\Z^d$ exhibits a phase transition: there exists $\lambda_c(\Z^d)\in (0,\infty)$ such that the process dies out if and only if $\lambda \leq \lambda_c$. The process is said to be subcritical, critical and supercritical respectively if $\lambda < \lambda_c,\;\lambda = \lambda_c$ and $\lambda > \lambda_c$. In the supercritical case, if the process is started with every site infected, then as $t \to \infty$ its distribution converges to a non-trivial invariant measure on $\{0,1\}^{\Z^d}$ called the upper invariant distribution; we denote it by $\bar \pi$.

Interestingly, this phase transition is also manifest for the contact process on finite subsets of the lattice. Let $\Gamma_n = \{1,\ldots, n\}^d$ and consider the contact process on $\Gamma_n$ with parameter $\lambda$ starting from all infected, $(\xi^{\Gamma_n}_t)$. As mentioned above, this process almost surely dies out. However, the expected extinction time grows logarithmically with $n$ when $\lambda < \lambda_c(\Z^d)$ and exponentially with $n$ when $\lambda > \lambda_c(\Z^d)$ (\cite{durliu}, \cite{tommeta}). In the latter case, metastability is said to occur, because the process persists for a long time in an equilibrium-like state which resembles the restriction of $\bar \pi$ to the box $\Gamma_n$, and eventually makes a quick transition to the true equilibrium - the absorbing state. In particular, if $(t_n)$ is a sequence of (deterministic) times that grows to infinity slower than the expected extinction times of $(\xi^{\Gamma_n}_t)_{t\geq 0}$, we have
\begin{equation}\frac{|\xi^{\Gamma_n}_{t_n}|}{n^d} \;\;\stackrel{n \to \infty}{\xrightarrow{\hspace*{0.8cm}}}\;\; \bar \pi\left(\left\{\xi\in\{0,1\}^{\Z^d}: \xi(0) = 1\right\}\right) \text{ in probability,}\label{eq:motivate}\end{equation}
where $|\cdot|$ denotes cardinality. This means that the density of infected sites in typical times of activity is similar to that of infinite volume.

The main theorem in this paper is a statement analogous to (\ref{eq:motivate}) for the contact process on a class of random graphs, namely the configuration model with power law degree distribution, as described in \cite{NSW} and \cite{remco}. Let us define these graphs. We begin with a probability measure $p$ on $\mathbb{N}$; this will be our degree distribution. We assume it satisfies
\begin{eqnarray}
&p(\{0, 1, 2\}) = 0;&\label{eq:ele1}\\[0.3cm]
&\text{for some }a > 2,\;{\displaystyle 0 < \liminf_{m \to \infty}m^a\;p(m) \leq \limsup_{m \to \infty}m^ap(m) < \infty.}&\label{eq:ele2}\end{eqnarray}
The first assumption, that $p$ is supported on integers larger than 2, is made to guarantee that the graph is connected with probability tending to 1 as $n \to \infty$ (\cite{CD}). The second assumption, that $p$ is a power law with exponent $a$, is based on the empirical verification that real-world networks have power law degree distributions; see \cite{Dur} for details.

For fixed $n \in \N$, we will construct the random graph $G_n = (V_n, E_n)$ on the set of $n$ vertices $V_n = \{v_1, v_2,\ldots, v_n\}$. To do so, let $d_1, \ldots, d_n$ be independent with distribution $p$. We assume that $\sum_{i=1}^nd_i$ is even; if it is not, we add 1 to one of the $d_i$, chosen uniformly at random; this change will not have any effect in any of what follows, so we will ignore it. For $1 \leq i \leq n$, we endow $x_i$ with $d_i$ \textit{half-edges} (sometimes also called \textit{stubs}). Pairs of half-edges are then matched so that edges are formed; since $\sum_{i=1}^n d_i$ is even, it is possible to match all half-edges, and an edge set is thus obtained. We choose our edge set $E_n$ uniformly among all edge sets that can be obtained in this way. We denote by $\P_{p,n}$ a probability measure under which $G_n$ is defined. If, additionally, a contact process with parameter $\lambda$ is defined on the graph, we write $\P^\lambda_{p,n}$.

\noindent \textbf{Remark.}
$G_n$ may have loops (edges that start and finish at the same vertex) and multiple edges between two vertices. As can be read from the generator in (\ref{eq:gen}), loops can be erased with no effect in the dynamics, and when vertices $x$ and $y$ are connected by $k$ edges, an infection from $x$ to $y$ (or from $y$ to $x$) is transmitted with rate $k\lambda$.

In \cite{CD}, Chatterjee and Durrett studied the contact process $(\xi^{V_n}_t)$ on $G_n$, and obtained the surprising result that it is ``always supercritical'': for any $\lambda > 0$, the extinction time grows quickly with $n$ (it was shown to be larger than a stretched exponential function of $n$). This contradicted predictions in the Physics literature to the effect that there should be a phase transition in $\lambda$ similar to the one we described for finite boxes of $\Z^d$. In \cite{MMVY}, the result of \cite{CD} was improved and the extinction time was shown to grow as an exponential function of $n$.

As already mentioned, our main theorem is concerned with the density of infected sites on the graph at times in which the infection is still active. The main motivation in studying this density is shedding some light into the mechanism through which the infection manages to remain active for a long time when its rate is very close to zero. In particular, our result shows that this mechanism depends on the value of $a$, the exponent of the degree distribution.
\begin{theorem}
\label{thm:main} There exist $c, C > 0 $ such that, for $\lambda>0$ small enough and $(t_n$) with $t_n \to \infty$ and $\log t_n = o(n)$, we have
$$\P_{p, n}^\lambda \left(c\rho_a(\lambda) \leq \frac{|\xi^{V_n}_{t_n}|}{n} \leq C\rho_a(\lambda) \right) \stackrel{n \to \infty}{\xrightarrow{\hspace*{0.8cm}}} 1,$$
where $\rho_a(\lambda)$ is given by
$$\rho_a(\lambda)  = \left\{\begin{array}{ll} \lambda^{\frac{1}{3-a}} &\text{if } 2 < a \leq 2\frac{1}{2}; \medskip\\\frac{\lambda^{2a-3}}{\log^{a-2}\left(\frac{1}{\lambda}\right)} &\text{if } 2\frac{1}{2} < a \leq 3;\medskip\\\frac{\lambda^{2a-3}}{\log^{2a-4}\left(\frac{1}{\lambda}\right)} &\text{if } a > 3.\end{array} \right.$$
\end{theorem}
This theorem solves the open problem of \cite{CD}, page 2337, for $a > 2$. For $2 < a \leq 3$, the result is new, as no estimates were previously available. For $a > 3$, it is an improvement of the non-optimal bounds that were obtained in \cite{CD}: there, it proved that for $a > 3$, the density is between $\lambda^{2a-3+\epsilon}$ and $\lambda^{a-1-\epsilon}$ for any $\epsilon > 0$, when $\lambda$ is small.

Very recently (\cite{remco2}), Dommers, Giardin\`a and van der Hofstad studied the ferromagnetic Ising model on random trees and locally tree-like random graphs with power law degree distributions (in particular, their results cover the class of graphs we consider in this paper). In their context, the above theorem, which is about the exponent of the metastable density of the contact process when $\lambda \to 0$, translates to studying the exponent of the magnetization of the Ising model as $\beta \to \beta_c$, where $\beta$ is the inverse temperature and $\beta_c$ its critical value. Similarly to the above theorem, they showed how this exponent depends on the exponent of the degree distribution, and this dependence also exhibits different regimes.

By a well-known property of the contact process called duality (see \cite{lig85}, Section III.4), for any $t > 0$ and $1 \leq i \leq n$ we have
\begin{equation}\P_{p,n}^\lambda\left(\xi^{V_n}_t(v_i) = 1 \right) = \P_{p,n}^\lambda\left(\xi^{v_i}_t\neq \varnothing \right).\label{eq:dualint}\end{equation}
On the right-hand side, we have the probability that the contact process started at $v_i$ is still active at time $t$. In the study of this probability, we are required to understand the local structure of $G_n$ around a typical vertex. This is given, in the limit as $n \to \infty$, by a two-stage Galton-Watson tree.

In order to precisely state this, let $q$ be the size-biased distribution associated to $p$, that is, the measure on $\N$ given by $q(m) = (\sum_{i \geq 0} i \cdot p(i))^{-1}\cdot  m\cdot  p(m)$ (note that the assumption that $a > 2$ implies that $\sum_{i \geq 0} i \cdot p(i) < \infty$). Let $\Q_{p,q}$ be a probability measure under which a Galton-Watson tree is defined with degree distribution of the root given by $p$ and degree distribution of all other vertices given by $q$. Note that, since $p(\{0,1,2\}) = q(\{0,1,2\}) =0$, this tree is infinite. We emphasize that we are giving the \textit{degree} distribution of vertices, and not their \textit{offspring} distribution, which is more commonly used for Galton-Watson trees. The following Proposition then holds; see \cite{CD} and Chapter 3 of \cite{Dur} for details. For a graph $G$ with vertex $x$ and $R > 0$, we denote by $B_G(x, R)$ the ball in $G$ with center $x$ and radius $R$.

\begin{proposition}
\label{lem:kGW}
For any $k \in \mathbb{N}$ and $R > 0$, as $n \to \infty$, the $k$ balls $B_{G_n}(v_1, R),\ldots, B_{G_n}(v_k, R)$ under $\P_{p,n}$ are disjoint with probability tending to 1. Moreover, they jointly converge in distribution to $k$ independent copies of $B_\T(o, R)$, where $\T$ is a Galton-Watson tree (with root $o$) sampled from the probability $\Q_{p,q}$.
\end{proposition}
(Obviously, in the above, there is nothing special about the vertices $v_1,\ldots, v_k$ and the result would remain true if, for each $n$, they were replaced by $v_{i_{n,1}}, \ldots, v_{i_{n,k}}$, with $1 \leq i_{n,1} < \cdots < i_{n,k} \leq n$).

With this convergence at hand, in \cite{CD}, the right-hand side of (\ref{eq:dualint}) (and then, by a second moment argument, the density of infected sites) is shown to be related to the probability of survival of the contact process on the random tree given by the measure $\Q_{p,q}$. In this paper, we make this relation more precise, as we now explain. We denote by $\Q_{p,q}^\lambda$ a probability measure under which the two-stage Galton-Watson tree described above is defined and a contact process of rate $\lambda$ is defined on the tree. Typically this contact process will be started from only the root infected, and will thus be denoted by $(\xi_t^o)_{t \geq 0}$. Let $\upgamma_p(\lambda)$ denote the survival probability for this process, that is,
$$\upgamma_p(\lambda) = \Q_{p,q}^\lambda \left(\xi^o_t \neq \varnothing \;\forall t \right).$$
As we will shortly discuss in detail, this quantity turns out to be positive for every $\lambda > 0$. Here we prove
\begin{theorem}
\label{thm:reduc}
For any $\lambda > 0,\;\epsilon > 0$ and $(t_n)$ with $t_n \to \infty$ and $\log t_n = o(n)$, we have
$$\P_{p,n}^\lambda \left(\left|\frac{|\xi^{V_n}_{t_n}|}{n} - \upgamma_p(\lambda) \right| > \epsilon\right) \stackrel{n \to \infty}{\xrightarrow{\hspace*{0.8cm}}} 0.$$
\end{theorem}

The above result was conjectured in \cite{CD}, page 2336, and is an improvement of their Theorem 1 (also note that we do not assume that $\lambda$ is small). Since the proof is essentially a careful rereading of the arguments in \cite{CD}, we postpone it to the Appendix. Our main focus in the paper will be finding the asymptotic behaviour of $\upgamma_p(\lambda)$ as $\lambda \to 0$:
\begin{proposition}
\label{prop:main}
There exist $c, C > 0$ such that $c\rho_a(\lambda)\leq \upgamma_p(\lambda) \leq C\rho_a(\lambda)$ for $\lambda$ small enough, where $\rho_a(\lambda)$ is the function defined in the statement of Theorem \ref{thm:main}.\end{proposition}

Theorem \ref{thm:main} immediately follows from Theorem \ref{thm:reduc} and Proposition \ref{prop:main}. Since we concentrate our efforts in proving Proposition \ref{prop:main}, in all the remaining sections of the paper (except the Appendix) we do not consider the random graph $G_n$. Rather, we study the contact process on the Galton-Watson tree started with the root infected, $(\xi^o_t)_{t \geq 0}$.

We will now describe the ideas behind the proof of the above proposition. 

\noindent $\bullet$ \textbf{Case} ${\bm a\;} {\bf > 3}$. Fix a small $\lambda > 0$. On the one hand, we show that a tree in which every vertex has degree smaller than $\frac{1}{8\lambda^2}$ is a ``hostile environment'' for the spread of the infection, in the sense that an infection started at vertex $x$ eventually reaches vertex $y$ with probability smaller than $(2\lambda)^{d(x,y)}$, where $d$ denotes graph distance (see Lemma \ref{lembound}). On the other hand, we show that if a vertex has degree much larger than $\frac{1}{\lambda^2}$, then it can sustain the infection for a long time; roughly, if $\deg(x) > \frac{K}{\lambda^2}$, then the infection survives on the star graph defined as $x$ and its neighbours for a time larger than $e^{cK}$ with high probability, where $c$ is a universal constant (see Lemma \ref{basic}). This is due to a ``bootstrap effect'' that occurs in this star, in which whenever $x$ becomes infected, it transmits the infection to several of its neighbours, and whenever it recovers, it receives the infection back from many of them.

Let us now call a vertex \textit{small} or \textit{big} depending on whether its degree is below or above the  $\frac{1}{8\lambda^2}$ threshold, respectively (this terminology is only used in this Introduction). It is natural to imagine that the infection can propagate on the infinite tree by first reaching a big vertex, then being maintained around it for a long time and, during this time, reaching another vertex of still higher degree, and so on. However, when $a > 3$, big vertices are typically isolated and at distance of the order of $\log \frac{1}{\lambda}$ from each other. This suggests the introduction of another degree threshold, which turns out to be of the order of $\frac{1}{\lambda^2}\log^2 \frac{1}{\lambda}$ -- let us call a vertex \textit{huge} if its degree is above this threshold. The point is that if a vertex is big but not huge, then although it maintains the infection for a long time, this time is not enough for the infection to travel distances comparable to $\log \frac{1}{\lambda}$, and hence not enough to reach other big sites. Huge vertices, in comparison, do maintain the infection for a time that is enough for distances of order $\log\left( \frac{1}{\lambda}\right)$ to be overcome.

With these ideas in mind, we define a key event $E^* = \{$the root has a huge neighbour $x^*$ that eventually becomes infected$\}$ (again, this terminology is exclusive to this Introduction). We think of $E^*$ as the ``best strategy'' for the survival of the infection. Indeed, in Section \ref{s:lower} we show that if $E^*$ occurs, then the infection survives with high probability and in Section \ref{s:upper} we show that every other way in which the infection could survive has probability of smaller order, as $\lambda \to 0$, than that of $E^*$. The probability of $E^*$ is roughly $q\left(\left[\frac{1}{\lambda^2}\log^2 \frac{1}{\lambda} ,\;\infty\right) \right)\cdot \lambda$, the first term corresponding to the existence of the huge neighbour and the second term to its becoming infected. Since $p(m) \asymp m^{-a}$ (that is, $m^a\cdot p(m)$ is bounded from above and below), we have $q(m) \asymp m^{-(a-1)}$ and $q([m,\infty)) \asymp m^{-(a-2)}$; using this, we see that, modulo constants, $q\left(\left[ \frac{1}{\lambda^2}\log^2 \frac{1}{\lambda} ,\;\infty\right) \right)\cdot \lambda$ is $\frac{\lambda^{2a-3}}{\log^{2a-4}\left(\frac{1}{\lambda}\right)}$, which is the definition of $\rho_a(\lambda)$ when $a > 3$.

\noindent $\bullet$ \textbf{Case} ${\bf 2\frac{1}{2} <}$ ${\bm a}$ ${\bf\leq 3}$. This case is very similar to the previous one. The main difference is that now, when growing the tree from the root, if we find a vertex of large degree $K$, then with high probability, it will have a child with degree larger than $K$ (or a grandchild in the case $a = 3$). As a consequence, defining small and big vertices as before, big vertices will no longer be in isolation, but will rather be close to each other. For this reason, once the infection reaches a big site, the distance it needs to overcome to reach another big site is small, and we have to modify the ``big-huge'' threshold accordingly. The new threshold is shown to be $\frac{1}{\lambda^2}\log \frac{1}{\lambda}$. The key event $E^*$ is then defined in the same way as before, and shown to have probability of the order of $\rho_a(\lambda)$.

\noindent $\bullet$ \textbf{Case} ${\bf 2 <}$ ${\bm a}$ ${\bf\leq 2\frac{1}{2}}$. In both previous cases, the ``bootstrap effect'' that we have described is crucial. Interestingly, it does not play an important role in the regime in which the tree is the largest, that is, $2 < a \leq 2\frac{1}{2}$. In this case, the survival of the infection does not at all depend on vertices of high degree sustaining the infection around them for a long time, as we now explain. We define a comparison process $(\eta_t)$ which is in all respects identical to the contact process $(\xi_t^o)$, with the only exception that once sites become infected and recover for the first time, they cannot become infected again. Thus, $(\xi^o_t)$ stochastically dominates $(\eta_t)$, that is, both processes can be constructed in the same probability space satisfying the condition $\xi_t^o(x) \geq \eta_t(x)$ for all $x$ and $t$. The process $(\eta_t)$ is much easier to analyse than the contact process, and we find a lower bound for the probability that it remains active at all times (and thus a lower bound for the survival probability of the contact process). We then give an upper bound for the survival probability of $(\xi^o_t)$ that matches that lower bound. This implies that, in the regime $2 < a \leq 2\frac{1}{2}$, as $\lambda \to 0$, the survival probabilities of $(\xi^o_t)$ and $(\eta_t)$ are within multiplicative constants of each other.

Finally, let us describe the organization of the paper. Section \ref{s:setup} contains a description of the graphical construction of the contact process and of the notation we use. In Section \ref{s:survestimate} we establish a lower bound for the survival time of the process on star graphs (Lemma \ref{basic}) and, as an application, a result that gives a condition for the process to go from one vertex $x$ of high degree to another vertex $y$ on a graph (Lemma \ref{lem:infNei}). In Section \ref{s:lower}, we use these results to prove the lower bound in Proposition \ref{prop:main}. In Section \ref{s:extestimate}, we give an upper bound on the probability that the process spreads on a tree of bounded degree (depending on $\lambda$), and in Section \ref{s:upper} we apply this to obtain the upper bound in Proposition \ref{prop:main}. In the Appendix, we prove Theorem \ref{thm:reduc}.

\section{Setup and notation}
\label{s:setup}
\subsection{Graphical construction of the contact process}
In order to fix notation, we briefly describe the graphical construction of the contact process. Let $G=(V,E)$ be a graph and $\lambda > 0$. We take a probability measure $P^\lambda_G$ under which we have a family $H$ of independent Poisson point processes on $[0,\infty)$ as follows:
$$\begin{aligned}
&\{D_x: x \in V\} \text{ with rate 1};\\
&\{D_{x,y}: \{x, y\} \in E\}\text{ with rate } \lambda.
\end{aligned}$$
The elements of the random sets $D_x$ are called recoveries, and those of the sets $D_{x,y}$ are called transmissions. The collection $H$ is called a graphical construction for the contact process on $G$ with rate $\lambda$. Given $x, y \in V$ and $0 \leq t_1 \leq t_2$, an infection path from $(x,t_1)$ to $(y, t_2)$ is a piecewise constant, right-continuous function $\gamma:[t_1,t_2] \to V$ satisfying $\gamma(t_1) = x,\; \gamma(t_2) = y$ and, for all $t$, \\
$\bullet\;$ if $\gamma(t-) \neq \gamma(t), \text{ then } \{\gamma(t-),\gamma(t)\} \in E \text{ and }t \in D_{\gamma(t-), \gamma(t)};$\\
$\bullet\;$  if $\gamma(t) = z,$ then $t \notin D_z$.\\
If such a path exists, we write $(x,t_1)\; \lra\; (y,t_2)$. Given $A,B \subset V$, $J_1, J_2 \subset [0,\infty)$, we write $A \times J_1 \;\lra\; B \times J_2$ if $(x, t_1) \;\lra\; (y, t_2)$ for some $x \in A,\; y\in B,\; t_1 \in J_1$ and $t_2 \in J_2$ with $t_1 \leq t_2$. Given a set $U \subset V$ with $A, B \subset U$, we say that $A \times J_1 \;\lra\; B \times J_2$ inside $U$ if $A \times J_1 \;\lra\; B \times J_2$ by an infection path that only visits vertices of $U$.

For $A \subset V$, by letting $\xi^A_t(x) = I_{\{A\times\{0\} \;\lra\;(x,t)\}}$ for each $t \geq 0$, we get a process $(\xi^A_t)_{t\geq 0}$ that has the same distribution as the contact process with initial configuration $I_A$, as defined by the generator (\ref{eq:gen}). A significant advantage of this construction is that, in a single probability space, we obtain contact processes with all initial configurations, $\left((\xi^A_t)_{t \geq 0} \right)_{A \subset V}$ with the property that for every $A$, $\xi^A_t = \cup_{x \in A}\;\xi^x_t$ and in particular, if $A \subset B$, we have $\xi^A_t \subset \xi^B_t$ for every $t$.

\subsection{Remarks on the laws $p$ and $q$}
Recall that our assumptions on the degree distribution $p$ are that $p(\{0,1,2\}) = 0$ and, for some $a > 2,\;c_0,\; C_0 > 0$ and large enough $k$, we have $c_0 k^{-a} < p(k) < C_0 k^{-a}$. The fact that $a > 2$ implies that
$\mu:=\sum_{k=1}^\infty kp(k) < \infty$
and that the size-biased distribution $q$ is well-defined. In case $a > 3$ we also have $\nu := \sum_{k=1}^\infty kq(k) < \infty$. We may and often will assume that the constants $c_0, C_0$ also satisfy, for large enough $k$,
\begin{eqnarray}
\label{c0a-1} &&p[k, \infty),\; q(k) \in (c_0k^{-(a-1)},\;C_0k^{-(a-1)});\\
\label{c0a-2} &&q[k, \infty) \in (c_0k^{-(a-2)},\;C_0k^{-(a-2)});\\
\label{c03-a} &&\sum_{k=0}^m k q(k) \in \left\{\begin{array}{ll}\;(c_0m^{3-a},\;C_0m^{3-a})&\text{if } 2 < a < 3;\smallskip\\\;(c_0\log(m),\;C_0\log(m)) &\text{if } a = 3.\end{array}\right.
\end{eqnarray}

\subsection{Notation}
For ease of reference, here we summarize our notation. Some of the points that follow were already mentioned earlier in the Introduction.

Given a graph $G$ and $\lambda > 0$, $P_G^\lambda$ denotes a probability measure for a graphical construction of the contact process on $G$ with rate $\lambda$. Under this measure, we can consider the contact process $(\xi^A_t)_{t\geq 0}$ on $G$ with any initial configuration $I_A$.

Unless otherwise stated, Galton-Watson trees are denoted by $\T$ and their root by $o$. Their degree distribution (or distributions in the case of two-stage trees) will be clear from the context. The probability measure is denoted $\Q_r$ if the degree distribution of all vertices is $r$ and $\Q_{r,s}$ if the root has degree distribution $r$ and other vertices have degree distribution $s$. If on top of the tree, a graphical construction for the contact process with rate $\lambda$ is also defined, we write $\Q^\lambda_r$ and $\Q^\lambda_{r,s}$.

$G_n$ denotes the random graph on $n$ vertices with fixed degree distribution $p$, as described above, and $\P_{p,n}$ a probability measure for a space in which it is defined. $\P_{p,n}^\lambda$ is used when a graphical construction with rate $\lambda$ is defined on the random graph.

The distribution $p$ has exponent $a$, as in (\ref{eq:ele2}), and its mean is denoted by $\mu$. If the sized-biased distribution $q$ has finite mean, this mean is denoted by $\nu$. Since $q$ may have infinite expectation, we may sometimes have to consider its truncation, that is, for $m > 0$, the law
\begin{equation}\overline q_m(k) =  \left\{\begin{array}{ll}q(m,\infty) &\text{if } k = 1;\\q(k),&\text{if }1 < k \leq m;\\ 0,&\text{if }k > m.\end{array}\right.\label{eq:defhatq}\end{equation}
(since $q$ is used as a \textit{degree} distribution for vertices of a tree, we set the minimum value of its truncation to 1).

On a graph $G$, $d(x, y)$ denotes graph distance and $B(x,R) = \{y: d(x, y) \leq R\}$. For a set $A$, we denote by $|A|$ the number of elements of $A$.


\section{A survival estimate on star graphs}
\label{s:survestimate}

We start looking at the contact process on star graphs, that is, graphs in which all vertices except a privileged one (called the hub) have degree 1. A first result to the effect that the contact process survives for a long time on a large star was Lemma 5.3 in \cite{BBCS}, which showed that, for a star $S$, if $\lambda$ is small and $\lambda^2|S|$ is larger than a universal constant, then the infection survives for a time that is exponential in $\lambda^2|S|$. The following result adds some more detail to that picture. 

\begin{lemma} \label{basic}
There exists $c_1 > 0$ such that, if $\lambda < 1$ and $S$ is a star with hub $o$, \medskip\\
$(i.)\; \displaystyle{P_S^\lambda\left(|\xi^o_1| > \frac{1}{4e}\cdot\lambda\deg(o)\right) \geq \frac{1}{e}(1-e^{-c_1 \lambda \deg(o)})};$\medskip\\
$(ii.)$ if $\lambda^2\deg(o) >  64e^2$ and $|\xi_0| >\frac{1}{16e} \cdot  \lambda\deg(o)$, then $P_S^\lambda\left(\xi_{e^{c_1 \lambda^2\deg(o)}} \neq \varnothing\right) \geq 1-e^{-c_1 \lambda^2\deg(o)};$\\[0.3cm]
$(iii.)\;$ as $|S| \to \infty,\;P_S^\lambda\left(\exists t:\; |\xi^o_t| >\frac{1}{4e}\cdot \lambda\deg(o) \right) \to 1$.
\end{lemma}
\begin{proof}
Here and in the rest of the paper, we use the following fact, which is a consequence of the Markov inequality: for any $n \in  \N$ and $r \in [0,1]$, if $X \sim \mathsf{Bin}(n, r)$ we have
\begin{eqnarray}\label{mark}\forall \alpha > 0 \;\exists \theta > 0:\; \P(|X-\E X| > \alpha n r) \leq e^{-\theta n r}.
\end{eqnarray}

For the event in $(i.)$ to occur, it is sufficient that there is no recovery at $o$ in $[0,1]$ and, for at least $\frac{\lambda}{4e}\deg(o)$ leaves, there is no recovery in $[0,1]$ and a transmission is received from $o$. Also using the inequality $1-e^{-\lambda} \geq \lambda/2$ for $\lambda < 1$, the probability in $(i.)$ is more than
$$\begin{aligned}&e^{-1} \cdot \P\left(\mathsf{Bin}\left(\deg(o),\; e^{-1}(1-e^{-\lambda})\right) >\frac{\lambda }{4e}\deg(o)\right) \\&\geq e^{-1} \cdot \P\left(\mathsf{Bin}\left(\deg(o),\; \frac{\lambda}{2e}\right) >\frac{\lambda }{4e}\deg(o)\right)\geq e^{-1}(1-e^{-c \lambda \deg(o)})\end{aligned}$$  
for some $c > 0$, by (\ref{mark}). $(i.)$ is now proved.

For $j \geq 0$, define
$$\begin{aligned}&\Gamma_j = \{y \in S\backslash \{o\}: D_y \cap [j, j+1] = \varnothing\},\\
&\Psi_j = \{y \in \Gamma_j: \xi^o_j(y) = 1\}.\end{aligned}$$
$\Psi_j$ is thus the set of leaves of the star that are infected at time $j$ and do not heal until time $j + 1$. For each $j \geq 0$, we will now define an auxiliary process $(Z^j_t)_{j \leq t \leq j+1}$. We put $Z^0_t \equiv 1$ and, for $j \geq 1$, put
\begin{itemize}
\item $Z^j_j = 0;$
\item for each $t \in [j, j+1]$ such that for some $y \in \Psi_j$ we have $t \in D_{y,o}$, put $Z^j_t = 1$;
\item for each $t \in [j, j+1] \cap D_o$, put $Z^j_t = 0$;
\item complete the definition of $Z^j_t$ by making it constant by parts and right-continuous.
\end{itemize}
It is then clear that 
\begin{equation} \label{hjEq1}Z^j_t \leq \xi^o_t(o) \; \forall j, t.\end{equation}

Consider the events, for $j \geq 0$:
$$\begin{aligned}
&A_{1,j} = \{|\Gamma_j| > \deg(o)/2e\};\\
&A_{2,j} = \{|\Psi_j| > \lambda \deg(o)/32e^2\};\\
&A_{3,j} = \left\{\int_j^{j+1} I_{\{Z_t = 1\}} \; dt > 1/2\right\};\\
&A_{4,j} = \left\{\left| \left\{\begin{array}{c}y \in \Gamma_j: \text{for some $t \in [j,j+1],$}\\Z^j_t = 1 \text{ and } t \in D_{o,y} \end{array}\right\}\right| > \frac{\lambda \deg(o)}{16e} \right\}.
\end{aligned}$$
Notice that, by (\ref{hjEq1}) and the definition of $A_{3,j}$,
\begin{equation}\label{hjEq2}\left\{\frac{1}{N} \int_0^N I_{\{\xi^o_t(o) = 1\}}\; dt > \frac{1}{2} \right\} \supset {\mathop\cap_{j=0}^{N-1}}\; A_{3,j} \qquad \forall N \in \N. \end{equation}

We have
\begin{equation}P_S^\lambda\left((A_{1,j})^c\right) \leq \P\left(\;\mathsf{Bin}(\deg(o), 1/e) \leq \deg(o)/2e\;\right) \leq e^{-\theta \deg(o)/e}. \label{hjB2}\end{equation}

We now want to bound $P_S^\lambda\left((A_{4,j})^c\;|\; A_{1,j} \cap A_{3,j}\right)$, for $j \geq 0$. By the definition of $(Z^j)$, the event $A_{1,j} \cap A_{3,j}$ depends only on $\xi^o_j$, $(D_y \cap [j, j+1])_{y \in S \backslash \{o\}}$ and $(D_{y,o} \cap [j,j+1])_{y \in S \backslash \{o\}}$. Therefore, conditioning on $A_{1,j} \cap A_{3,j}$ does not affect the law of $(D_{o,y} \cap [j, j+1])_{y \in S \backslash \{x\}}$, the set of arrows from the hub to the leaves at times in $[j, j+1]$. We thus have
\begin{flalign}\nonumber &P_S^\lambda\left((A_{4,j})^c\;|\; A_{1,j} \cap A_{3,j}\right)\leq \P\left(\mathsf{Bin}(\deg(o)/2e,\;1-e^{-\lambda/2}) \leq \lambda \deg(o)/16e\right)& \\&\qquad\qquad\qquad \leq \P\left(\;\mathsf{Bin}(\deg(o)/2e,\; \lambda/4)\leq \lambda \deg(o)/16e \right)) \leq e^{-\theta \lambda \deg(o)/8e}  \label{hjB3} \quad \forall j \geq 0.& \end{flalign}

Let us now bound $P^\lambda_S((A_{3,j})^c\;|\; A_{2,j})$. Define the continuous-time Markov chains $(Y_t)_{t \geq 0},\;(Y'_t)_{t \geq 0}$ with state space $\{0,1\}$ and infinitesimal parameters
$$\begin{array}{ll}
q_{01} = \frac{\lambda^2\deg(o)}{32e^2},&q_{10} = 1; \medskip\\
q'_{01} = 1,& q'_{10} = \frac{32e^2}{\lambda^2\deg(o)}.
\end{array}$$
Now, if $32e^2/(\lambda^2\deg(o)) < 1/2$ we have
$$\begin{aligned}
P^\lambda_S\left((A_{3,j})^c\;|\; A_{2,j}\right) \leq \P\left(\int_0^1 I_{\{Y_t = 0\}}\; dt \geq \frac{1}{2}\right) &= \P\left(\frac{32e^2}{\lambda^2\deg(o)}\int_0^{\frac{\lambda^2\deg(o)}{32e^2}}I_{\{Y'_t = 0\}}\;dt \geq \frac{1}{2}\right).
\end{aligned}$$
Denoting by $\pi$ the invariant measure for $Y'$, we have $\pi_0 = \frac{\frac{32e^2}{\lambda^2\deg(o)}}{1+\frac{32e^2}{\lambda^2\deg(o)}} < \frac{1}{3}$. Then, by the large deviations principle for Markov chains (see for example \cite{dembzeit}), we get
\begin{equation} P^\lambda_S\left((A_{3,j})^c\;|\; A_{2,j}\right) \leq e^{-c\lambda^2\deg(o)}\label{hjB4}\end{equation}
for some $c > 0$.

Finally, for any $j \geq 1$ we have
\begin{equation} \label{hjB5}P^\lambda_S\left((A_{2,j})^c\;|\;A_{4,j-1}\right) \leq \P\left(\mathsf{Bin}(\lambda \deg(o)/16e,\; 1/e) \leq \lambda \deg(o)/32e^2\right) \leq e^{-\theta \lambda^2\deg(o)/16\lambda e^2}\end{equation}
and similarly, $P^\lambda_S\left((A_{2,0})^c\right) \leq e^{-\theta \lambda^2\deg(o)/16\lambda e^2}$.
Putting together (\ref{hjEq2}), (\ref{hjB2}), (\ref{hjB3}), (\ref{hjB4}) and (\ref{hjB5}), we get the desired result.

Statement $(iii.)$ can be proved by similar (and simpler) arguments than $(ii.)$, so for brevity we omit a full proof.
\end{proof}

As an application of the previous result, for two vertices $x$ and $y$ of a connected graph, we give a condition on $\deg(x)$ and $d(x,y)$ that guarantees that, with high probability, the infection is maintained long enough around $x$ to produce a path that reaches $y$.

\begin{lemma}
\label{lem:infNei}
There exists $\lambda_0 > 0$ such that, if $0 < \lambda < \lambda_0$, the following holds. If $G$ is a connected graph and $x, y$ are distinct vertices of $G$ with $$\deg(x) > \frac{3}{c_1}\frac{1}{\lambda^2}\log\left(\frac{1}{\lambda}\right) \cdot d(x, y) \quad \text{ and }\quad \frac{|\xi_0\;\cap\; B(x,1)|}{\lambda \cdot|B(x,1)|} > \frac{1}{16e},$$ then
$$P_G^\lambda\left(\exists t: \frac{|\xi_t\;\cap\; B(y,1)|}{\lambda \cdot|B(y,1)|} > \frac{1}{16e} \right)>1-2e^{-c_1\lambda^2\deg(x)}.$$
\end{lemma}
\begin{proof}
Let $r = 2d(x,y)$ and $L = \lfloor \frac{\exp(c_{1} \lambda^2 \deg(x))}{r} \rfloor$. Define the event $$A^2_1 = \{\forall s \leq Lr,\; \exists z \in B(x,1): \xi^x_s(z) = 1 \}.$$
By Lemma \ref{basic} we have $P_{G}^\lambda(A^2_1) \geq 1 - e^{-c_1\lambda^2 \deg(x)}.$

Further define the events 
$$A^2_{2,i} = \{\exists z \in B(x,1): \xi^x_{ir}(z) = 1\}, \quad i= 0, \ldots, L-1$$
so that $A^2_1 \subset \cap_{i=1}^{L-1} A^2_{2,i}$.
On $A^2_{2,i}$, we can choose $Z_i \in B(x,1)$ such that $\xi^x_{ir}(Z_{i}) = 1$ and a sequence $\gamma_{i,0} = Z_i,\; \gamma_{i,1}, \ldots, \gamma_{i,k_i} = y$ such that $d(\gamma_{i,j}, \gamma_{i,j+1}) = 1 \; \forall j$ and $k_i \leq d(x,y)+1 \leq r$. Define
$$A^2_{3,i} = A^2_{2,i} \cap \left\{ \exists s \in [ir,\; (i+1)r - 1): (Z_i, ir) \;\lra\; (y,s)\text{ and } \frac{|\xi_{s+1} \;\cap \;B(y,1)|}{\lambda \cdot|B(y,1)|} > \frac{1}{16e} \right\},$$
We claim that 
\begin{equation}\label{eqnPathToy}P_{G}^\lambda\left( A^2_{3,i}\;|\;A^2_{2,i},\;(\xi_t)_{0\leq t \leq ir}\right) \geq \left(e^{-1}(1-e^{-\lambda})\right)^r\cdot e^{-1}(1-e^{-c_1\lambda\deg(y)}).\end{equation}
To see this, note that an infection path from $(Z_i, ir)$ to $\{y\} \times [ir, \; (i+1)r - 1)$ can be obtained by imposing that, for $0 \leq j < k_i$, there is no recovery in $\{\gamma_{i,j}\} \times [ir + j,\; ir + j + 1)$ and at least one transmission from $\gamma_{i,j}$ to $\gamma_{i, j+1}$ at some time in $[ir + j,\; ir+j + 1)$. This explains the term $(e^{-1}(1-e^{-\lambda}))^r$ in the right-hand side of (\ref{eqnPathToy}). The other term comes from Lemma \ref{basic}(i.).

The right-hand side of (\ref{eqnPathToy}) is larger than $\left(\frac{\lambda}{3}\right)^{r} \cdot \frac{c_1\lambda}{2e}$ when $\lambda$ is small.
We then have
$$\begin{aligned}P_{G}^\lambda\left(A^2_1 \cap (\cup_{i=0}^{L-1}\; A^2_{3,i})^c\right) &\leq P_{G}^\lambda\left((\cap_{i=0}^{L-1} \; A^2_{2,i}) \cap (\cup_{i=0}^{L-1}\; A^2_{3,i})^c\right)\\
&\leq P_{G}^\lambda\left((\cap_{i=0}^{L-1} \; A^2_{2,i}) \cap (\cup_{i=0}^{L-2}\; A^2_{3,i})^c\right) \cdot \left(1-(c_1\lambda/2e)\left(\lambda/3\right)^{r}\right)\\
&\leq P_{G}^\lambda\left((\cap_{i=0}^{L-2} \; A^2_{2,i}) \cap (\cup_{i=0}^{L-2}\; A^2_{3,i})^c\right) \cdot \left(1-(c_1\lambda/2e)\left(\lambda/3\right)^{r}\right)\end{aligned}$$
and iterating, this is less than
$$\begin{aligned} \left(1-\frac{c_1\lambda}{2e}\left(\frac{\lambda}{3}\right)^{r} \right)^{L} \leq \exp\left\{-\frac{c_1\lambda}{2e}\left( \frac{\lambda}{3}\right)^{2d(x,y)} \cdot \frac{1}{2d(x,y)} \cdot \exp\left\{c_1\lambda^2\deg(x) \right\} \right\},
\end{aligned}$$
which is smaller than $e^{-c_1\lambda^2\deg(x)}$ if $\lambda$ is small enough, since $\deg(x) > \frac{3}{c_1}\frac{1}{\lambda^2}\log\frac{1}{\lambda}$. This completes the proof.
\end{proof}


\section{Proof of Proposition \ref{prop:main}: lower bounds}
\label{s:lower}
Given a random graph $G$ (which will be either a Galton-Watson tree or the random graph $G_n$), a vertex $x$ of $G$ and $R,K > 0$, define the event
\begin{equation}
\mathcal{M}(x,R,K) = \{\exists y: d(x,y) \leq R,\; \deg(y) > K\}.\label{eq:defM}
\end{equation}
We will need the following simple result on Galton-Watson trees.
\begin{lemma}\label{lem:auxFinal}
If $2 < a \leq 3$, then ${\displaystyle\;\; \liminf_{K \to \infty}\; \Q_q\left(\;\mathcal{M}(o, 2, K\log K)\; \left|\;\deg(o) = K \right.\right) > 0}$.
\end{lemma}
\begin{proof}
Assume $\deg(o) = K$ and define
$$A=\left\{|\{x: d(o,x) = 2\}| > \frac{c_0}{2}K\log K\right\}.$$
Let $z_1, \ldots, z_K$ be the neighbours of the root and $Z_i = \deg(z_i)-1$ for $1 \leq i \leq K$, so that $\sum Z_i = |\{x: d(o,x) = 2\}|$. Note that the law of the $Z_i$ is given by $k \mapsto q(k+1)$, thus stochastically dominates the distribution $k \mapsto \hat q(k):=\overline q_{K}(k+1)$, where $\overline q_{K}$ is the truncation of $q$, as defined in (\ref{eq:defhatq}). Let $Y_1, Y_2, \ldots$ be i.i.d. with distribution $\hat q(k)$. We then have, by (\ref{c03-a}),
$$\E(Y_1) > c_0 \log(K), \qquad \text{Var}(Y_1) \leq \sum_{k\leq K} k^2q(k+1) \leq \bar C_0 K$$
where $\bar C_0 > 0$ is a constant that depends only on $p$. Then,
$$\begin{aligned}
&\Q_{q}(A\;|\;\deg(o) = K) = \Q_q\left(\left.\sum_{k \leq K} Z_k > \frac{c_0}{2} \;K\log K\;\right|\;\deg(o) = K\right) \geq \P\left(\sum_{k \leq K} Y_k > \frac{c_0}{2} \;K\log K\right) \\&\qquad\qquad\qquad\qquad> 1 - \P\left(\left|\sum_{k \leq K} Y_k - K\cdot \E(Y_1)\right| > \frac{c_0}{2} K\log K \right) > 1 - \frac{\bar C_0 \cdot K^2}{\left(\frac{c_0}{2} K\log K\right)^2} \stackrel{K \to \infty}{\xrightarrow{\hspace*{0.8cm}}} 1.
\end{aligned}$$
Now, if $A$ occurs, then there are at least $\frac{c_0}{2}K\log K$ vertices at distance 2 from $o$. Each of these vertices has degree larger than $K\log K$ with probability $q(K\log K,\;\infty) \geq c_0(K \log K)^{-(a-2)} \geq c_0(K \log K)^{-1}$, since $a \leq 3$. We thus get
$$\begin{aligned}&\Q_{q}(\mathcal{M}(o, 2, K\log K)\;|\;\{\deg(o) = K\}\;\cap\;A) \geq 1 - \left(1-c_0(K\log K)^{-1}\right)^{\frac{c_0}{2}K\log K} \\&> 1 - \exp \left\{-c_0\left(K\log K\right)^{-1} \cdot \frac{c_0}{2}K\log K \right\} = 1 - \exp\left\{-\frac{(c_0)^2}{2}\right\}.\end{aligned}$$
This completes the proof.
\end{proof}

Again assume that $G$ is a random graph; also assume we have a graphical construction for the contact process with parameter $\lambda$ on $G$. Let 
$$\chi_t = \{y: (x,0) \lra (y,t) \text{ inside } B(x,R)\}.$$
Then, define
\begin{equation}
\mathcal{N}(x,R,K) = \left\{\exists y, t: d(x,y) \leq R,\;\deg(y) > K,\;\frac{|\chi_t \;\cap\;B(y,1)|}{|B(y,1)|} > \frac{\min(\lambda,\lambda_0)}{16e} \right\},\label{eq:defN}
\end{equation}
where $\lambda_0$ is as in Lemma \ref{lem:infNei}. In words, in the contact process started from $x$ infected, a proportion larger than $\frac{\min(\lambda, \lambda_0)}{16e}$ of the neighbours of $y$ become infected at some time $t$, and this occurs through infection paths contained in the ball $B(x, R)$.

In this subsection, we will assume that $\lambda < \lambda_0$, as in Lemma \ref{lem:infNei}, and often will state conditions that require $\lambda$ to be sufficiently small.

\subsection{Case $2\frac{1}{2} < a \leq 3$}\label{ss:lower}
Define 
$$\begin{array}{lll}
K_1 = \frac{12}{c_1} \frac{1}{\lambda^2} \log \left(\frac{1}{\lambda}\right), &K_2 = \frac{18a}{c_1\log 2} \frac{1}{\lambda^2} \log^2 \left(\frac{1}{\lambda}\right)\,&K_i = \frac{1}{\lambda^3} + i - 3,\; i\geq 3;\medskip\\
R_1 = 1,&R_2 = 3,&R_i = \lceil a\log_2 K_i\rceil,\; i \geq 3,
\end{array}$$
where $c_1$ is as in Lemma \ref{basic}. We will show that, for some $c > 0$ and $\lambda$ small enough,
\begin{eqnarray}
&&\Q_{p,q}\left({\mathop \cap_{i=1}^\infty} \; \mathcal{M}(o, R_i, K_i)\right) > c\; \left(\frac{\lambda^2}{\log \left(\frac{1}{\lambda}\right)}\right)^{a-2} \text{ and}\label{eqn:eqRK1}\\
&&\Q_{p,q}^\lambda \left({\mathop \cap_{i=1}^\infty}\; \mathcal{N}(o, R_i, K_i)\;\left|\;{\mathop \cap_{i=1}^\infty}\;\mathcal{M}(o, R_i, K_i) \right.\right) > c\lambda. \label{eqn:eqRK2}
\end{eqnarray}
Since $\{\xi^o_t \neq \varnothing\;\forall t\} \supset \cap_{i=1}^\infty\;\mathcal{N}(o, R_i, K_i)$, these inequalities will give us the desired result.

To prove (\ref{eqn:eqRK2}) we assume $\cap_{i=1}^\infty\;\mathcal{M}(o, R_i, K_i)$ occurs and let $y_1, y_2, \ldots$ denote sites with $\deg(y_i) > K_i$ and $d(o, y_i) \leq R_i$ (so that $d(y_i, y_{i+1}) \leq 2R_{i+1}$). With probability $\frac{\lambda}{1+\lambda}$, the root infects its neighbour $y_1$ before recovering (unless the root itself is equal to $y_1$, in which case this probability is 1). Then, by Lemma \ref{basic}$(i.)$, with probability larger than $e^{-1}(1-e^{-c_1\lambda\deg(y_1)})$, we have $\frac{|\xi^o_t \;\cap\;B(y_1,1)|}{\lambda \cdot|B(y_1,1)|} > \frac{1}{16e}$ for some $t > 0$, so that $\mathcal{N}(o, R_1, K_1)$ occurs.  Since, for each $i$,
$$\deg(y_i) > K_i > \frac{3}{c_1}\left(\frac{1}{\lambda^2} \log \frac{1}{\lambda}\right)\cdot 2R_{i+1} \geq \frac{3}{c_1} \left(\frac{1}{\lambda^2} \log \frac{1}{\lambda}\right)\cdot d(y_i, y_{i+1}),$$
we can repeatedly use Lemma \ref{lem:infNei} to guarantee that, with probability larger than $1-2\sum_{i=1}^\infty\;e^{-c_1\lambda^2K_i}$, for each $i$ there exists $t > 0$ such that $\frac{|\xi^o_t \;\cap\;B(y_i,1)|}{\lambda \cdot|B(y_i,1)|} > \frac{1}{16e}$. This shows that
$$\Q_{p,q}^\lambda\left({\mathop \cap_{i=1}^\infty \;\mathcal{N}(o, R_i, K_i)} \left| {\mathop \cap_{i=1}^\infty \;\mathcal{M}(o, R_i, K_i)}\right.\right) > \frac{\lambda}{1 + \lambda} \cdot (e^{-1}(1-e^{-c_1\lambda K_1}))\cdot \left(1 - 2\sum_{i=1}^\infty e^{-c_1\lambda^2K_i}\right) > \frac{\lambda}{3}$$
when $\lambda$ is small enough.

We now turn to (\ref{eqn:eqRK1}). By (\ref{c0a-2}) we have
\begin{equation}\label{eq:lowboundEq31}\Q_{p,q}(\mathcal{M}(o, R_1, K_1)) \geq c_0\left(\frac{12}{c_1} \frac{1}{\lambda^2}\log \frac{1}{\lambda} \right)^{-(a-2)}.\end{equation}
On the event $\mathcal{M}(o, R_1, K_1) = \mathcal{M}(o, 1, K_1)$, again let $y_1$ denote a vertex in $B(o, 1)$ with degree larger than $K_1$. By Lemma \ref{lem:auxFinal}, we have
$$\Q_{p,q}\left(\mathcal{M}(y_1, 2, K_1\log K_1)\;|\;\mathcal{M}(o, 1, K_1) \right) > \bar c$$
for some $\bar c > 0$ that does not depend on $\lambda$. Since $K_1 \log K_1 > K_2$ for $\lambda$ small enough, this implies that
\begin{equation}\label{eq:lowboundEq32}\Q_{p,q}\left(\mathcal{M}(o, 3, K_2)\;|\;\mathcal{M}(o, R_1, K_1) \right) > \bar c.\end{equation}

To give a lower bound for the probability of $\mathcal{M}(o, R_i, K_i)$ when $i \geq 3$, we observe that there are at least $2^{R_i-1}$ vertices at distance $R_i-1$ from the root, by the fact that the degrees of all vertices are at least 3. Thus,
$$\begin{aligned} \Q_{p,q}(\mathcal{M}(o, R_i, K_i)) &\geq 1 - (1 - c_0K_i^{-(a-2)})^{2^{R_i-1}} \\&\geq 1 -\exp\left\{-\frac{c_0}{2}\cdot K_i^{-(a-2)}\cdot K_i^{a} \right\} = 1 - \exp\left\{-\frac{c_0}{2}\left(\frac{1}{\lambda^3} + i - 3\right)^2\right\}.\end{aligned}$$
We then get
\begin{equation}\label{eq:RKfinal}
\Q_{p,q}\left(\left({\mathop \cap_{i=3}^\infty}\; \mathcal{M}(o, R_i, K_i)\right)^c \right) < \sum_{i=3}^\infty e^{-\frac{c_0}{2}\left(\frac{1}{\lambda^3} + i -3\right)^2}
\end{equation}
and, as $\lambda \to 0$, the right-hand side converges to 0 faster than any power of $\lambda$. Inequality (\ref{eqn:eqRK1}) now follows from (\ref{eq:lowboundEq31}), (\ref{eq:lowboundEq32}) and (\ref{eq:RKfinal})

\subsection{Case $a > 3$}
This case is very similar to the previous one, only simpler. The proof can be repeated with the constants now given by
$$\begin{array}{ll}K_1 = \frac{12a}{\log2} \cdot \frac{1}{\lambda^2} \log^2 \left(\frac{1}{\lambda}\right),& K_i = \frac{1}{\lambda^3}+i-2,\;i\geq 2;\medskip\\R_1 = 1,&R_i = \lceil a\log_2K_i \rceil,\; i \geq 2. \end{array}$$
It is thus shown that
$$\Q_{p,q}\left({\mathop \cap_{i=1}^\infty}\mathcal{M}(o, R_i, K_i) \right) > c\;\left(\frac{\lambda^2}{\log^2\left(\frac{1}{\lambda}\right)}\right)^{a-2}\text{and }\;\; \Q_{p,q}^\lambda \left( {\mathop \cap_{i=1}^\infty} \mathcal{N}(o, R_i, K_i)\;\left|\;{\mathop \cap_{i=1}^\infty} \mathcal{M}(o, R_i, K_i) \right. \right) > c\lambda.$$

\subsection{Case $2 < a \leq 2\frac{1}{2}$}
Recall the definition of $\hat q$ in (\ref{eq:defhatq}). We will show that
\begin{equation}\label{eq:afred2}\Q^\lambda_{\hat q}\left(\xi^o_t \neq \varnothing \; \forall t\right) > c \lambda^{\frac{1}{3-a} - 1}. \end{equation}
This will give the desired result since
$$\Q^\lambda_{p,q}\left(\xi^o_t \neq \varnothing \; \forall t\right) \geq \frac{\lambda}{1+\lambda} \cdot\Q^\lambda_{\hat q}\left(\xi^o_t \neq \varnothing \; \forall t\right).$$

In order to study the contact process $(\xi^o_t)_{t \geq 0}$ on $\T$ (a tree sampled from $\Q_{\hat q}$) and started from only the root infected, we introduce a comparison process $(\eta_t)_{t \geq 0}$, started from the same initial condition. $(\eta_t)$
will be a modification of the contact process: sites will become permanently set to value $0$ the first time (if ever) that they return to value $0$ after having taken value $1$. Consequently, sites cannot infect sites closer to the root than themselves.

More precisely, $(\eta_t)_{t \geq 0}$ is defined as follows. Suppose we are given a tree $\T$ and a graphical construction $\{(D_x)_{x \in \T},\; (D_{x,y})_{x,y \in \T,\; x \sim y}\}$ with parameter $\lambda > 0$. Let $\sigma_o = \inf D_o$ be the first recovery time at the root and set $\eta_t(o) = I_{[0, \sigma_o)}(t)$ for all $t \geq 0$. Now assume $(\eta_t(x))_{t \geq 0}$ has been defined for all $x$ at distance $m$ or less from the root, and fix $y$ with $d(o,y) = m+1$. Let $z$ be the parent of $y$, that is, $d(o, z) = m$ and $d(z, y) = 1$. Let $\tau_y = \inf\left(\{t: \eta_t(z) = 1\} \cap D_{z, y}\right)$ and, if $\tau_y < \infty$, let $\sigma_y = \inf\left([\tau_y, \infty) \cap D_y \right)$. Now, if $\tau_y < \infty$, set $\eta_t(y) = I_{[\tau_y, \sigma_y)}(t)$ for all $t$ and otherwise set $\eta_t(y) = 0$ for all $t$.

Define
$$X_m:=\left|\{z:d(o,z) = m, \exists  t < \infty \text{ with } \eta_t(z) = 1 \}\right|$$
for $m=0,1,2,\ldots$ Then $(X_m)_{m\geq0}$ is a branching process and is in principle easy to analyze. We start with the following lemma, which gives
a lower bound for the probability $\Q_{\hat q}^\lambda(X_1\geq
k)$.

\begin{lemma}\label{l:1}
There exists $c_{2.1}$ such that, for $\lambda \in (0,1)$ and all $k \geq 1$,
$$\Q_{\hat q}^\lambda(X_1\geq k)\geq c_{2.1}\left(\lambda/k\right)^{a-2}.$$
\end{lemma}
\begin{proof} For $k\geq1$, we have
$$\begin{aligned}\Q_{\hat q}^\lambda(X_1 \geq k) &\geq \Q_{\hat q}^\lambda\left(X_1 \geq k,\; \sigma_o \geq 1,\; \deg(o) \geq \frac{2k}{1-e^{-\lambda}}\right)\\&\geq \Q_{\hat q}^\lambda \left(\sigma_o \geq 1,\; \deg(o) \geq \frac{2k}{1-e^{-\lambda}} \right)\cdot \P\left(\mathsf{Bin}\left(\left\lceil\frac{2k}{1-e^{-\lambda}}\right \rceil,\;1-e^{-\lambda} \right) \geq k\right)\\
&=e^{-1}\cdot \hat q\left[\frac{2k}{1-e^{-\lambda}},\;\infty \right)\cdot \P\left(\mathsf{Bin}\left(\left \lceil\frac{2k}{1-e^{-\lambda}}\right \rceil,\;1-e^{-\lambda} \right) \geq k\right) \\&\geq C \cdot \hat q\left[\frac{2k}{1-e^{-\lambda}},\;\infty \right) \geq c_{2.1} \left(\frac{\lambda}{k} \right)^{a-2}
.\end{aligned}$$\end{proof}

As a consequence of the above result, $X_1$ has infinite expectation, so, with positive probability, $X_n \to \infty$ as $n \to \infty$. We define the generating function for the law of $X_1$:
$$\Psi_\lambda(s) = \sum_{n=0}^\infty \Q_{\hat q}^\lambda(X_1 = n) \cdot s^n \qquad (s \in (0,1]).$$
We can use Lemma \ref{l:1} to get the following estimate for $\Psi_\lambda(s)$, where the infection parameter $\lambda > 0$ is fixed.
\begin{lemma}
\label{l:2} There exists $c_{2.2} > 0$ such that, for $\lambda \in (0,1)$ and $s \in [1/2, 1]$,
$$\Psi_\lambda(s) \leq 1 - c_{2.2}(\lambda(1-s))^{a-2}.$$
\end{lemma}
\begin{proof}
By monotonicity of $s^m$ in $m$, we have for any positive integer $k$
$$\Psi_\lambda(s) \leq \sum_{i=0}^k \Q_{\hat q}^\lambda(X_1 = i) + s^k \sum_{i = k+1}^\infty \Q_{\hat q}^\lambda(X_1 = i) = 1 - \Q_{\hat q}^\lambda(X_1 \geq k)\cdot (1-s^k).$$
We choose $k$ equal to $\left \lfloor\frac{1}{1-s}\right \rfloor$ which gives the desired inequality, since $s \mapsto 1 - s^{\lfloor \frac{1}{1-s}\rfloor}$ is bounded away from zero for $s \in [1/2, 1]$.
\end{proof}

From Lemma \ref{l:2} we can easily get the following
\begin{corollary}
There exists $c_{2.3} > 0$ such that, for $\lambda > 0$ small enough,
$$\Q_{\hat q}^\lambda(X_n \neq 0 \; \forall n) \geq c_{2.3}\;\lambda^{\frac{a-2}{3-a}}.$$
\end{corollary}
\begin{proof}
We know that $X_1$ has infinite expectation and so from the standard theory of branching processes (see e.g. \cite{durprob}), we have that the survival probability $\beta$ satisfies
$$\beta  = 1 - \Psi_\lambda(1-\beta).$$
By Lemma \ref{l:2}, the right-hand side is larger than $c_{2.2}(\lambda \beta)^{a-2}$, so $\beta > c_{2.2}(\lambda\beta)^{a-2}$, so $\beta >(c_{2.2})^\frac{1}{3-a}\cdot \lambda^\frac{a-2}{3-a}.$
\end{proof}

Since $\frac{a-2}{3-a}=\frac{1}{3-a} - 1$ and $\{\xi^o_t \neq \varnothing \; \forall t\} \supset \{X_n \neq 0\; \forall n\}$, (\ref{eq:afred2}) is now proved.


\section{Extinction estimates on star graphs and trees}
\label{s:extestimate}
Our main objective in this section is to establish estimates that allow us to say, under certain conditions, that the contact process does not spread too much and does not survive too long. In Lemma \ref{lembound}, we obtain upper bounds for the probability of existence of certain infection paths on finite trees of bounded degree. In Lemma \ref{lem:extStar}, we obtain a result for star graphs that works in the reverse direction as that of Lemma \ref{basic}: with high probability, the contact process on a star graph $S$ does not survive for longer than $e^{C\lambda^2|S|}$, for some large $C > 0$.

\begin{lemma} \label{lembound}
\noindent Let $\lambda <\frac{1}{2}$ and $T$ be a finite tree with maximum degree bounded by $\frac{1}{8\lambda^2}$. Then, for any $x, y \in T$ and $0 < t < t'$,\medskip\\
$(i.)\;P_T^\lambda\left(\;(x,0) \;\lra\; \{y\} \times \R_+\;\right) \leq (2\lambda)^{d(x,y)};$\medskip\\
$(ii.)\; P_{T}^\lambda\left(\;(x,0) \;\lra\; \{y\} \times [t, \infty)\;\right) \leq (2\lambda)^{d(x,y)}\cdot e^{-t/4};$\medskip\\
$(iii.)\;P_{T}^\lambda \left( \xi^T_t \ne \emptyset \right) \ \leq \ |T|^2\cdot e^{-t/4};$\medskip\\
$(iv.)\;P_{T}^\lambda\left(\;\{x\} \times [0, t] \;\lra\; \{y\} \times \R_+\;\right) \leq (t+1)\cdot(2\lambda)^{d(x,y)}$;\medskip \\
If $x \neq y$,\medskip\\
$(v.)\;P_{T}^\lambda\left(\;\exists \ell < \ell':\;(x,0) \;\lra\; (y, \ell) \;\lra\; (x,\ell') \;\right) \leq (2\lambda)^{2d(x,y)};$\medskip\\
$(vi.)\;P_{T}^\lambda\left(\;\exists \ell:\;\{x\}\times [0,t] \;\lra\; (y, \ell) \text{ and } (y, \ell) \;\lra\; \{x\} \times [\ell, \infty)\;\right) \leq (t+1)\cdot(2\lambda)^{2d(x,y)}.$
\end{lemma}

\begin{proof} 
\noindent $(i.)$ For $u > 0$, let $M_u = \sum_{z \in T}\; \xi^x_u(z) \cdot (2\lambda)^{d(z,y)}$. We claim that $(M_u)_{u \geq 0}$ is a supermartingale. To check this, notice that, for fixed $u \geq 0$ and $\xi \in \{0,1\}^T$,
\begin{eqnarray}\nonumber&&\frac{d}{dr} E_{T}^\lambda\left(M_{u+r} \;|\; \xi_u = \xi\right)\vert_{r=0+} = \sum_{\substack{z \in T:\;\xi(z) = 1}} \left(\left(\lambda \cdot \sum_{\substack{w: w \sim z,\;\xi(w) = 0}} (2\lambda)^{d(w,y)}\right) - (2\lambda)^{d(z,y)}\right)\\
&&\qquad\qquad\nonumber\leq \sum_{\substack{z \in T:\;\xi(z) = 1}} \left(2^{d(z,y)-1} \cdot \lambda^{d(z,y)} + \frac{1}{8\lambda^2} \cdot 2^{d(z,y) + 1} \cdot \lambda^{d(z,y) + 2} - (2\lambda)^{d(z,y)}\right)\\
&&\nonumber \qquad\qquad\leq \sum_{\substack{z \in T:\;\xi(z)=1}} \lambda^{d(z,y)}\left(2^{d(z,y) -1} + 2^{d(z,y)-2} - 2^{d(z,y)}\right) \\
&&\qquad \qquad= -\frac{1}{4}  \sum_{z \in T} \xi(z) \cdot (2\lambda)^{d(x,y)}.\label{eq:14mart}\end{eqnarray}

Then, if $0 \leq s < u$,
$$\begin{aligned}&\frac{d}{dr}E_{T}^\lambda\left(M_{u+r}\;|\; \xi_{s'}: 0 \leq s' \leq s\right)\big \vert_{r=0+} \\&\quad= \sum_\xi \frac{d}{dr}E_{T}^\lambda\left(M_{u+r}\;|\; \xi_u = \xi\right)\vert_{r=0+} \cdot P_{T}^\lambda\left(\xi_u = \xi\;|\; \xi_{s'}:0 \leq s' \leq s\right) < 0.\end{aligned}$$
In addition, the function $u \in [s,\infty) \mapsto E_{T}^\lambda\left(M_u\;|\;\xi_{s'}: 0 \leq s'\leq s\right)$ is continuous. Consequently, it is decreasing, so $E_{T}^\lambda(M_u\;|\; \xi_{s'}:0\leq s'\leq s) \leq M_s.$

Now, let $\tau = \inf\{u > 0: \xi^x_u(y) = 1\}$. By the optional sampling theorem (which may be applied since $M$ is a c\`adl\`ag supermartingale), we get
$$\begin{aligned}P_{T}^\lambda((x,0) \lra \{y\} \times \R_+) &= P_{T}^\lambda(\tau < \infty) \\&\leq E_{T}^\lambda(M_\tau;\; \tau < \infty) \leq E_{T}^\lambda(M_0) = (2\lambda)^{d(x,y)}.\end{aligned}$$

\noindent $(ii.)$ Since by (\ref{eq:14mart}) for any $u$ we have $$\frac{d}{dr} E_{T}^\lambda\left(M_{u+r}\;|\; \xi_u\right)\vert_{r=0+} \leq -\frac{1}{4}M_u,$$ the process $\tilde M_u = e^{u/4} \cdot M_u$ is a supermartingale. Now define $\sigma_t = \inf\{u \geq t: \xi^x_u(y) = 1\}$. The optimal sampling theorem gives
$$\begin{aligned}P_{T}^\lambda\left((x,0) \;\lra\; \{y\}\times [t,\infty)\right) &\leq e^{-t/4}\cdot E_{T}^\lambda\left(\tilde M_{\sigma_t}\cdot I_{\{\sigma_t < \infty\}}\right) \\&\leq e^{-t/4}\cdot E_{T}^\lambda\left(\tilde M_0\right) = e^{-t/4}\cdot (2\lambda)^{d(x,y)},\end{aligned}$$
completing the proof.\medskip

\noindent $(iii.)$ Again applying the optimal sampling theorem to the supermartingale $(\tilde M_u)$ defined above, we get
\begin{equation} \label{eqnxry}P_{T}^\lambda\left(\xi^x_u(y) = 1\right) \leq e^{-u/4} \cdot (2\lambda)^{d(x,y)} \qquad \forall u. \end{equation}
Applying (\ref{eqnxry}) and the fact that $\lambda < 1/2$,
$$P_{T}^\lambda\left(\xi^T_t \neq \emptyset\right) \leq \sum_{x,y\in T}\; P_{T}^\lambda\left(\xi^x_t(y) = 1\right) \leq |T|^2\cdot e^{-t/4}.$$
\noindent $(iv.)$ For $u > 0$ and $z \in T$, define $\zeta_u(z) = I_{\{\{x\} \times [0, t] \;\lra\; (z,u)\}}$. $(\zeta_u)_{u \geq 0}$ is thus a process that evolves as $(\xi_u)_{u \geq 0}$, with the difference that site $x$ is ``artificially'' kept at state 1 until time $t$. Next, define for $u > 0$
$$N_u = \max(t+1-u,\; 1) \cdot \zeta_u(x) \cdot (2\lambda)^{d(x,y)} + \sum_{z \neq x} \zeta_u(z)\cdot (2\lambda)^{d(z,y)}.$$
We claim that $(N_u)_{u \geq 0}$ is a supermartingale. As in the previous parts, this is proved from
\begin{equation}\label{eqnCad} \frac{d}{dr}E_{T}^\lambda\left(N_{u+r}\;|\;\zeta_u\right)\vert_{r=0+} < -\frac{1}{4}N_u < 0. \end{equation}
In case $u \geq t$, (\ref{eqnCad}) is proved exactly as in the first computation in the proof of part $(i.)$. In case $u < t$, we note that
$$\begin{aligned}
&\frac{d}{dr} E_{T}^\lambda\left((t+1-u-r)\cdot \zeta_{u+r}(x) \cdot (2\lambda)^{d(x,y)}\;|\; \zeta_u\right)\big|_{r=0+} = -(2\lambda)^{d(x,y)},
\end{aligned}$$
so the same computation can again be employed and (\ref{eqnCad}) follows. The result is now obtained from the optional sampling theorem and the fact that $N_0 \equiv t+ 1$. \medskip

\noindent $(v.)$ The proofs of $(v.)$ and $(vi.)$ are similar but $(v.)$ is easier, so we only present $(vi.)$.\medskip

\noindent $(vi.)$ For $u \geq 0$ and $z \in T$, define
$$\begin{aligned}&\eta_u(z) = I\{\{x\} \times [0,t] \lra (z,u) \text{ by a path that does not pass by }y\};\\
&\eta'_u(z) = I\{\{x\} \times [0, t] \lra (z,u) \text{ by a path that passes by }y\}.
\end{aligned}$$
Notice that, in particular, $\eta_u(x) = 1 \; \forall u \leq t$, $\eta_u(y) = 0 \; \forall u$ and
$$\left\{\begin{array}{c}\exists t':\;\{x\}\times [0,t] \lra (y, t')\\ \text{ and } (y, t') \lra \{x\} \times [t', \infty)\end{array} \right\} = \{\exists s:\;\eta'_s(x) = 1\}.$$
Also define, for $u \geq 0$,
$$\begin{aligned}L_u = \max\left((t+1-u),\; 1\right)\cdot \eta_u(x) \cdot(2\lambda)^{2d(x,y)} &+ \sum_{\substack{z\in T:\\z \neq x}}\;\eta_u(z)\cdot(2\lambda)^{d(z,y) + d(y,x)}\\
&+ \sum_{z\in T}\; \eta'_u(z) \cdot (2\lambda)^{d(z,x)}. \end{aligned}$$
Proceeding as in the previous parts (and again treating separately the cases $u<t$ and $u \geq t$), we can show that $(L_u)_{u \geq 0}$ is a supermartingale. The result then follows from the optional sampling theorem (consider the stopping time $\inf\{s:\; \eta'_s(x) = 1\}$) and the fact that $L_0 = (t+1)(2\lambda)^{2d(x,y)}$.
\end{proof}

\begin{lemma}
\label{lem:extStar}
If $\lambda < 1/4$ and $S$ is a star,
$$P^\lambda_S\left(\xi^S_{3\log\left(\frac{1}{\lambda}\right)} = \varnothing\right) \geq \frac{1}{4}\;e^{-16\lambda^2|S|}.$$
\end{lemma}
\begin{proof}
Let $(\zeta^S_{t})_{t\geq 0}$ be the process with state space $\{0,1\}^S$, starting from full occupancy, and with the same dynamics as that of contact process, with the only difference that recovery marks at the hub $o$ have no effect, so that $o$ is permanently in state 1. $(\xi^S_t)$ and $(\zeta^S_t)$ can obviously be jointly constructed with a single graphical construction, with the property that $\xi^S_t \leq \zeta^S_t$ for all $t$. Also note that the processes $\{\zeta^S_t(x): x \in S\}$ are independent and, if $x \neq o$, the function $t \mapsto P_S^\lambda(\zeta^S_t(x) = 1)$ is a solution of $f'(t) = \lambda(1-f(t)) - f(t)$, so
$$P^\lambda_S\left(\zeta^S_t(x) = 1\right) = \frac{1}{1+\lambda}\left(\lambda + e^{-(1+\lambda)t}\right).$$

Let $\sigma = \inf D_o \cap \left[\log \frac{1}{\lambda},\;\infty \right)$ be the first recovery time at the hub after time $\log \frac{1}{\lambda}$. Also define the events
$$\begin{aligned}
&B^1_1 = \left\{\sigma < 2\log \frac{1}{\lambda}\right\};\\
&B^1_2 = \left\{|\zeta^S_\sigma| \leq 4\lambda|S|\right\};\\
&B^1_3 = \left\{\begin{array}{c}\text{For all } x \in \zeta^S_\sigma,\; D_x \cap \left[\sigma,\; \sigma + \log \frac{1}{\lambda}\right] \neq \varnothing\medskip\\ \text{and } \inf\left(D_x \cap [\sigma,\;\infty)\right) < \inf\left(D_{x,o} \cap [\sigma,\;\infty)\right) \end{array}\right\}.
\end{aligned}$$
We then have $\left\{\xi^S_{3\log\frac{1}{\lambda}} = \varnothing\right\} \supset  B^1_1 \cap B^1_2 \cap B^1_3$. To see this, assume that the three events occur. By the definition of $\sigma$, we have $\xi^S_\sigma(o) = 0$ and $|\xi^S_\sigma| \leq |\zeta^S_\sigma| < 4\lambda S$ and every vertex that is infected at this time recovers without reinfecting the root by time $\sigma + \log(1/\lambda) < 3\log(1/\lambda)$.

We have
$$P_S^\lambda(B^1_1) \geq P^\lambda_S\left(D_o \cap \left[\log \frac{1}{\lambda},\;2\log \frac{1}{\lambda}\right] \neq \varnothing\right)\geq 1 - e^{-\log\frac{1}{\lambda}} = 1 -\lambda.$$
For $x \neq o,\; P_S^\lambda\left(\zeta^S_\sigma(x) = 1\right) \leq \frac{1}{1+\lambda}\left(\lambda + e^{-(1+\lambda)\log\frac{1}{\lambda}}\right) \leq 2\lambda,$ so
$$P_S^\lambda(B^1_2) \geq \P\big(\mathsf{Bin}(|S|,\;2\lambda) \leq 4\lambda|S| \big) \geq 1/2.$$
Also, for $x \neq o$, 
$$P_S^\lambda\left(\begin{array}{c}D_x \cap \left[\sigma,\; \sigma + \log \frac{1}{\lambda}\right] \neq \varnothing \text{ and}\medskip\\ \inf\left(D_x \cap [\sigma,\;\infty)\right) < \inf\left(D_{x,o} \cap [\sigma,\;\infty)\right) \end{array} \right) \geq 1 - e^{-\log\frac{1}{\lambda}}-\frac{\lambda}{1+\lambda} > 1 - 2\lambda,$$
so
$$P^\lambda_S\left(B^1_3\;|\;B^1_1 \cap B^1_2\right) \geq (1-2\lambda)^{4\lambda|S|} \geq e^{-2\cdot 2\lambda \cdot 4\lambda|S|} = e^{-16\lambda^2|S|}$$
since $1-\alpha \geq e^{-2\alpha}$ for $\alpha < 1/2$.

In conclusion,
$$P^\lambda_S\left(\xi^S_{3\log \frac{1}{\lambda}} = \varnothing\right) \geq \P^\lambda_S\left(B^1_3\;|\;B^1_1 \cap B^1_2 \right)\cdot P^\lambda_S\left(B^1_1 \cap B^1_2\right) \geq e^{-16\lambda^2|S|}\cdot\left(1-\lambda - \frac{1}{2}\right) \geq \frac{1}{4}\;e^{-16\lambda^2|S|}.$$
\end{proof}

Applying the above result and Lemma \ref{lembound}, we get a bound on the probability of extinction of the contact process on trees where, one vertex apart, degrees are bounded by $\frac{1}{8\lambda^2}$.
\begin{lemma}
\label{lem:extTree}
For $\lambda > 0$ small enough, the following holds. If $T$ is a tree with root $o,\;|T| < \frac{1}{\lambda^3}$ and $\deg(x) \leq \frac{1}{8\lambda^2}$ for all $x \neq o$, then
$$P^\lambda_T\left(\xi^T_{100\log\frac{1}{\lambda}} = \varnothing\right) \geq \frac{1}{8}\;e^{-16\lambda^2\deg(o)}.$$
\end{lemma}
\begin{proof}
Let $S$ be the star graph containing $o$ and it neighbours, $T' = T\backslash \{o\}$ be the disconnected graph obtained by removing $o$ and all edges incident to it from $T$ and $L = \frac{100}{3}\log\frac{1}{\lambda}$. We introduce three basic comparison processes, all generated with the same graphical construction on $T$ that is used to define $(\xi^T_t)_{t\geq 0}$.\medskip\\
$\bullet\;\left(\xi^{T,1}_t\right)_{t \geq 0}$ is the contact process on $T'$ started from full occupancy, that is,
$$\xi^{T',1}_{t} = {\left\{x: T'\times \{0\} \;\lra\;(x,t) \text{ inside } T'\right\}};$$
$\bullet\;\left(\eta^S_t\right)_{t \geq L}$ is the contact process on $S$, beginning from full occupancy at time $L$, that is,
$$\eta^S_t = \left\{x: S \times L \;\lra\;(x,t) \text{ inside } S\right\};$$
$\bullet\;\left(\xi^{T,2}_t\right)_{t \geq 2L}$ is the contact process on $T'$ started from full occupancy at time $2L$, that is,
$$\xi^{T',2}_t = \left\{x: T' \times 2L \;\lra\; (x, t) \text{ inside } T' \right\}.$$
The event $\left\{\xi^T_{3L} = \varnothing \right\}$ contains the intersection of the following events:
$$\begin{aligned}
&B^2_1 = \left\{\xi^{T,1}_L = \varnothing \right\};\qquad B^2_2 = \left\{\eta^S_{2L} = \varnothing\right\};\qquad 
B^2_3 = \left\{\xi^{T,2}_{3L} = \varnothing \right\};\\
&B^2_4 = \left\{\nexists(x,s): d(o,x) \geq 2,\; o \times [0,\;3L] \;\lra\; (x,s) \;\lra\; \{o\}\times [s,\;\infty)\right\}.
\end{aligned}$$
Let us prove this. If $B^2_1$ occurs, then for any $s \geq L$ and $x\in T,\;\xi^T_s(x) = 1$ implies that $\{o\}\times [0,s] \;\lra\;(x,s)$. Thus, if $B^2_1 \cap B^2_4$ occurs, then for any $s \geq L,\;\eta^S_s(o) = 0$ implies $\xi^T_s(o) = 0$ so that, if $B^2_1 \cap B^2_2 \cap B^2_4$ occurs, we have $\xi^T_s(o) = 0$ for $s \geq 2L$. It then follows that $\xi^T_{3L} =\varnothing$ if all four events occur.

We now note that the four events are decreasing with respect to the partial order on graphical constructions defined by setting, for graphical constructions $H$ and $H',\; H \prec H'$ if $H'$ contains more transmissions and less recoveries than $H$. Thus, by the FKG inequality, $P^\lambda_T(\cap_{i=1}^4 B^2_i) \geq \prod_{i=1}^4 P^\lambda_T(B^2_i)$. 

By Lemma \ref{lembound}$(vi.)$, we have $P^\lambda_T(B^2_4) \geq 1 - (3L+1)\cdot(2\lambda)^4\cdot \frac{1}{\lambda^3}$. Also applying Lemma \ref{lembound}$(iii.)$ and Lemma \ref{lem:extStar}, it is then easy to verify that, for $\lambda$ small, $P^\lambda_T(B^2_1)\cdot P^\lambda_T(B^2_2) \cdot P^\lambda_T(B^2_3) > \frac{1}{2}$.
\end{proof}


\section{Proof of Proposition \ref{prop:main}: upper bounds}
\label{s:upper}

We now want to apply the estimates of the previous section in proving the upper bound of Proposition \ref{prop:main}. Since Lemma \ref{lembound} must be applied to finite trees, our first step is defining truncations of infinite trees: given the distance threshold $r$ and the size threshold $m$, vertices of  degree larger than $m$ and vertices at distance $r$ from the root will be turned into leaves.

Let $r, m \in \N$ and $T$ be a tree with root $o$. Define the $r,m$-truncated tree 
$$\overline T_{r,m} = \{o\} \cup \left\{x \in T: d(o, x) \leq r,\; \deg(y) \leq m \; \forall y \text{ in the geodesic from $o$ to $x$},\; y \notin \{o, x\}\right\}.$$
Also define, for $1 \leq i < r$,
$$S^{i}_{r,m}(T) =\left\{\begin{array}{c}x \in T: d(o,x) = i,\;\deg(x) > m,\\ \deg(y) \leq m \; \forall y \text{ in the geodesic from $o$ to $x$},\; y \notin \{o, x\}\end{array}\right\}$$
and, finally,
$$S^r_{r,m}(T) = \{x \in \overline T_{r,m}: d(o, x) = r \}.$$
We want to think of $\overline T_{r,m}$ as the result of inspecting $T$ upwards from the root until generation $r$ so that, whenever a vertex $x$ of degree larger than $m$ is found, the whole subtree that descends from it is deleted, so that $x$ becomes a leaf.

Note that, if $\T$ is a tree sampled from the probability $\Q_{p,q}$, then $\overline \T_{r,m}$ is a Galton-Watson tree of $r$ generations in which the degree distribution of the root is $p$ and that of other vertices is $\overline q_m(k)$, as in (\ref{eq:defhatq}).
In particular, using (\ref{c0a-2}) and (\ref{c03-a}), for $1 \leq i < r$ we have the upper bound
\begin{equation}
\label{eq:expSRM1}
\E_{\Q_{p,q}}(\;|S^i_{r,m}|\;) \leq \mu\cdot \left(\sum_{k=1}^m kq(k)\right)^{i-1} \cdot q(m, \infty) \leq \left\{\begin{array}{ll}C_0\mu\cdot (C_0m^{3-a})^{i-1}\cdot m^{-(a-2)} & \text{if } 2 < a < 3;\medskip\\ C_0\mu\cdot (C_0\log m)^{i-1}\cdot m^{-1} &\text{if } a = 3;\medskip\\C_0\mu\cdot \nu^{i-1}\cdot m^{-(a-2)} &\text{if } a> 3. \end{array}  \right.
\end{equation}
Similarly, for $1 \leq i \leq r$,
\begin{equation}
\label{eq:expSRM2}
\E_{\Q_{p,q}}\left(\left|\left\{\begin{array}{c}x \in \overline \T_{r,m}:\\ d(o,x) = i\end{array}\right\}\right|\right) \leq \mu \cdot \left(\sum_{k=1}^m kq(k)\right)^{i-1} \leq \left\{\begin{array}{ll}C_0\mu\cdot (C_0m^{3-a})^{i-1} &\text{if } 2 < a< 3;\\C_0\mu\cdot (C_0\log m)^{i-1}&\text{if } a=3;\\C_0\mu\cdot \nu^{i-1}&\text{if } a > 3. \end{array}\right.
\end{equation}



Throughout the following subsections, we will take degree and distance thresholds, $M$ and $R$, which will depend on $\lambda$ and $a$, as follows:
$$M = \left\{ \begin{array}{cl}\left(\frac{1}{8C_0\lambda}\right)^{\frac{1}{3-a}} &\text{if } 2 < a \leq 2\frac{1}{2};\medskip\\
\frac{1}{8\lambda^2} &\text{if } a > 2\frac{1}{2};
\end{array}\right. \qquad\qquad R = \left\{\begin{array}{cl}\big \lceil \frac{2\log 4}{3-a} \log\left(\frac{1}{\lambda}\right) \big \rceil &\text{if } 2 < a \leq 2\frac{1}{2};\medskip\\ \big \lceil {\frac{2a+1}{2a-5}} \big \rceil&\text{if } 2\frac{1}{2} < a \leq 3;\medskip \\ 2a+1 &\text{if } a > 3. \end{array} \right.$$

\subsection{Case $2 < a \leq 2\frac{1}{2}$}
\label{ss:a>2}
The treatment of this regime is very simple. We start defining the event that the root has degree above $M$,
$$B^3_1 = \left\{\deg(o) > M\right\},$$
and the event that the root has degree below $M$ and the infection reaches a leaf of the truncated tree,
$$B^3_2 = \left\{\deg(o) \leq M,\; (o,0) \lra \left(\mathop{\cup}_{i=1}^R S_{R,M}^{i} \right) \times \R_+ \text{ inside } \overline \T_{R,M}\right\}.$$
We observe that $\{\xi^o_t \neq \varnothing \; \forall t\} \subset B^3_1 \cup B^3_2$. We wish to show that both events have probability smaller than $C\rho_a(\lambda)$ for some universal constant $C$. For the first event, this is immediate:
$$\Q_{p,q}(B^3_1) \leq C_0(8\lambda)^{\frac{a-1}{3-a}} < \lambda^{\frac{1}{3-a}}$$
when $\lambda$ is small. For the second,
\begin{equation}\label{eq:sepdeg} \Q_{p,q}^\lambda(B^3_2) \leq \sum_{i=1}^R \sum_{k=1}^\infty \Q_{p,q}^\lambda \left(\begin{array}{c}(o,0) \lra S^i_{R,M} \times \R_+\medskip \\ \text{ inside } \overline \T_{R,M} \end{array} \left| \begin{array}{c}\deg(o) \leq M,\medskip\\ |S^i_{R,M}| = k \end{array} \right. \right) \cdot \Q_{p,q}(\;|S^i_{R,M}| = k\;).  \end{equation}
Since the degrees of vertices of $\overline \T_{R,M}$ are bounded by $\left(\frac{1}{8\lambda C_0}\right)^{\frac{1}{3-a}} < \frac{1}{8\lambda^2}$, Lemma \ref{lembound}$(i.)$ implies that the conditional probability inside the sum is less than $k(2\lambda)^i$. (\ref{eq:sepdeg}) is thus less than $\sum_{i=1}^R (2\lambda)^i \; \E_{\Q_{p,q}}(\;|S^i_{R,M}|\;)$. Using (\ref{eq:expSRM1}) and (\ref{eq:expSRM2}), we have
$$\begin{aligned}
&\E_{\Q_{p,q}}(\;|S^i_{R,M}|\;) \leq C_0\mu\cdot (C_0M^{3-a})^{i-1} \cdot M^{-(a-2)} \text{ for } i < R \text{ and }\\
&\E_{\Q_{p,q}}(\;|S^R_{R,M}|\;) \leq C_0\mu \cdot (C_0M^{3-a})^{R-1}.
\end{aligned}$$
Thus,
\begin{equation}\sum_{i=1}^R(2\lambda)^i\;\E_{\Q_{p,q}}(\;|S^i_{R,M}|\;) \leq C_0\mu\cdot 2\lambda\left(M^{-(a-2)}\sum_{i=1}^{R-1} (2\lambda\cdot C_0 M^{3-a})^{i-1} + (2\lambda\cdot C_0 M^{3-a})^{R-1} \right).\nonumber\end{equation}
By the definition of $M$, $2\lambda \cdot C_0M^{3-a} < \frac{1}{2}$. Using the definition of $R$, we also have $$(2\lambda \cdot C_0 M^{3-a})^{R-1} < \left(\frac{1}{2}\right)^{R-1} < \left(\frac{1}{\lambda}\right)^\frac{2-a}{3-a}.$$
In conclusion, 
$$\Q_{p,q}^\lambda(B^3_2) \leq 2 C_0 \mu \cdot \lambda\left(C_0 \lambda^{\frac{a-2}{3-a}}\sum_{i=1}^{R-1} \left(\frac{1}{2}\right)^{i-1} + \left(\frac{1}{\lambda}\right)^{\frac{2-a}{3-a}} \right) < C\lambda^{1+\frac{a-2}{3-a}} = C\lambda^\frac{1}{3-a}.$$

\subsection{Case $a > 2\frac{1}{2}$}
\label{ss:a>212}
If we merely repeated the computation of the previous subsection with the new value of the threshold $M = \frac{1}{8\lambda^2}$ and the same events, that would yield a correct upper bound, but it would not be optimal in this case. So we will need to consider more events, taking a closer look at the truncated tree and ways in which the infection can leave it.

Our first two events are similar to those of the previous subsection:
$$\begin{aligned}
&B^4_1 = \{\deg(o) > M \};\medskip\\
&B^4_2 = \left\{\deg(o) \leq M,\;(o,0)\lra \left(\mathop{\cup}_{i=2}^R S_{R,M}^{i}\right)\times \R_+\text{ inside } \overline \T_{R,M}\right\}\end{aligned}$$
The difference to the previous subsection is that $B^4_2$ only includes leaves at distance two or more from the root. We will have to treat separately the  leaves neighbouring the root. We first consider the case in which there are at least two leaves neighbouring the root and at least one of them becomes infected:
$$B^4_3 = \{\deg(o) \leq M,\; |S_{R,M}^{1}| \geq 2,\;(o,0) \;\lra\; S_{R,M}^{1} \times \R_+ \text{ inside } \overline \T_{R,M}\}.$$
Next, if the root has only one neighbour that is a leaf (that is, if $|S_{R,M}^1|=1$), then call this neighbour $o^*$. Let us distinguish two ways in which $o^*$ may receive the infection initially present from $o$. We say that $o^*$ becomes infected \textit{directly} if a transmission from $o$ to $o^*$ occurs  before the first recovery time at $o$. We say that $o^*$ becomes infected \textit{indirectly} if there are infection paths starting at $(o,0)$ and ending at $\{o^*\} \times [0,\infty)$, but all of them must visit at least one vertex different from $o$ and $o^*$. We then define
$$\begin{aligned}&B^4_4 = \{\deg(o) \leq M,\; |S_{R,M}^{1}| = 1,\;\text{$o^*$ becomes infected indirectly}\},\\[0.25cm]
&B^4_5 = \left\{\begin{array}{c}\deg(o) \leq M,\; |S_{R,M}^{1}| = 1,\;\text{there exists $t^*$ such that}\medskip\\\text{ $o^*$ becomes infected directly at time $t^*$ and}\medskip\\(o^*,t^*)\;\lra\;\T\times\{t\}\text{ for all } t\geq t^*\end{array}\right\}.
\end{aligned}$$
Thus, in event $B^4_5$, there is a transmission from $o$ to $o^*$ at some time $t^*$ before the first recovery at $o$, and the infection generated from this transmission then survives for all times in the (non-truncated) tree $\T$.

The reason we make the distinction between $o^*$ becoming infected directly or indirectly is subtle; let us explain it. In our treatment of $B^4_5$, we will re-root the tree at $o^*$ and study the distribution of this re-rooted tree, so that we can find estimates for the infection that is transmitted from $(o^*, t^*)$. This study will be possible because, when we are told that a direct transmission has occurred, we only obtain information concerning the recovery process $D_o$ and the transmission process $D_{o,o^*}$, so the distribution of the degrees of other vertices in the tree is unaffected. In contrast, if the transmission is indirect, we have information concerning the portion of the tree that descends from $o$ through vertices different from $o^*$, so the study of the re-rooted tree is compromised and we have to follow a different approach.

We now have $\{\xi^o_t \neq \varnothing \; \forall t\} \subset \cup_{i=1}^5 B^4_i$. We wish to show that $\Q_{p,q}^\lambda (B^4_i) < C\rho_a(\lambda)$ for each $i$.\medskip

\noindent\textbf{1) Event $B^4_1$.} As in the previous section, we have $$\Q_{p,q}^\lambda(B^4_1) \leq C_0M^{-(a-1)} \leq 2C_0\cdot (8\lambda)^{2(a-1)} < \rho_a(\lambda)$$ when $\lambda$ is small.\medskip

\noindent \textbf{2) Event $B^4_2$.} As in our treatment of $B^3_2$ in the previous section, we have
\begin{equation}\Q_{p,q}^\lambda(B^4_2) \leq \sum_{i=2}^R (2\lambda)^i \cdot \E_{\Q_{p,q}}(\;|S^i_{R,M}|\;) = 2\lambda \sum_{i=2}^R (2\lambda)^{i-1} \cdot \E_{\Q_{p,q}}(\;|S^i_{R,M}|\;).\label{eq:3casesExp0}\end{equation}
Using (\ref{eq:expSRM1}), for $2 \leq i < R$, we have 
\begin{equation}\nonumber (2\lambda)^{i-1}\cdot \E_{\Q_{p,q}}(\;|S^i_{R,M}|\;)\label{eq:3casesExp}\leq \left\{\begin{array}{ll}CM^{-(a-2)}\cdot (C'\lambda M^{3-a})^{i-1} & \\\leq C\lambda^{2a-4}\cdot(C'\lambda^{2a-5})^{i-1}  &\text{if } 2\frac{1}{2} < a < 3;\medskip\\ CM^{-(a-2)}\cdot(C'\lambda \log M)^{i-1}&\\\leq C\lambda^{2a-4}\cdot(C'\lambda \log \frac{1}{\lambda})^{i-1}&\text{if } a = 3;\medskip\\CM^{-(a-2)}\cdot\nu^{i-1}&\\\leq C\lambda^{2a-4}\cdot \nu^{i-1}&\text{if } a > 3. \end{array}\right.
\end{equation}
We then have
\begin{equation}\nonumber \sum_{i=2}^{R-1} (2\lambda)^{i-1}\cdot \E_{\Q_{p,q}}(\;|S^{i}_{R,M}|\;) \leq \left\{\begin{array}{ll}C\lambda^{2a-4} \cdot \lambda^{2a-5}\cdot \sum_{i=2}^\infty (C'\lambda^{2a-5})^{i-2}&\text{if } 2\frac{1}{2} < a < 3;\medskip\\ C\lambda^{2a-4}\cdot \lambda \log\frac{1}{\lambda}\cdot \sum_{i=2}^\infty (C'\lambda\log\frac{1}{\lambda})^{i-2}&\text{if } a = 3;\medskip\\ C\lambda^{2a-4}\cdot \lambda \cdot \sum_{i=2}^\infty (C'\lambda)^{i-2}&\text{if } a > 3, \end{array} \right.  \end{equation}
for constants $C, C'$ that do not depend on $\lambda$. Then,
\begin{equation}\label{eq:3casesExp2}
\sum_{i=2}^{R-1} (2\lambda)^{i-1}\cdot \E_{\Q_{p,q}}(\;|S^{i}_{R,M}|\;) \leq \lambda^{2a - 4 + \delta}
\end{equation}
when $\lambda$ is small enough, for some $\delta > 0$ that depends on $a$ but not on $\lambda$.

For $i = R$, using (\ref{eq:expSRM2}) we get
\begin{equation}\nonumber
(2\lambda)^{R-1}\cdot \E_{\Q_{p,q}}(\;|S^{R}_{R,M}|\;) \leq \left\{\begin{array}{ll}C\left(C'\lambda^{2a-5}\right)^{R-1}&\text{if } 2\frac{1}{2} < a < 3; \medskip\\ C\left(C' \lambda \log \frac{1}{\lambda}\right)^{R-1} &\text{if } a = 3;\medskip\\ C(C' \lambda)^{R-1} &\text{if } a > 3. \end{array} \right.
\end{equation}
By the choice of $R$ in each case, when $\lambda$ is small we get
\begin{equation}
\label{eq:3casesExp3} 
(2\lambda)^{R-1}\cdot \E_{\Q_{p,q}}(\;|S^{R}_{R,M}|\;) \leq \lambda^{2a}.
\end{equation}
Using (\ref{eq:3casesExp2}) and (\ref{eq:3casesExp3}) in (\ref{eq:3casesExp0}), we conclude that, if $\lambda$ is small,
$$\Q_{p,q}^\lambda(B^4_2) \leq C\lambda^{2a-3+\delta} < \rho_a(\lambda).$$

\noindent \textbf{3) Event $B^4_3$.} We bound
\begin{equation}\Q_{p,q}^\lambda(B^4_3)  \leq \sum_{k=3}^\infty p(k) \cdot 2\lambda \cdot \E_{\Q_{p,q}}\left(\;|S^1_{R,M}| \cdot I_{\{|S^1_{R,M}| \geq 2\}} \; \big| \deg(o) = k\;\right)\label{eq:ests1rm}.\end{equation}
Under $\Q_{p,q}(\;\cdot |\deg(o) = k\;)$, $|S^1_{R,M}|$ is $\mathsf{Bin}(k, q(M,\infty))$. If $X \sim \mathsf{Bin}(n,p)$, then
\begin{equation}\label{eq:binn2}\E(X\cdot I_{\{X\geq 2\}}) = np - np(1-p)^{n-1} =np(1-(1-p)^{n-1}) < (np)^2,\end{equation}
since, by Bernoulli's Inequality, $(1-p)^{n-1} > 1-(n-1)p >1 -np$. Using the bound (\ref{eq:binn2}) for $k \leq M$ and the bound $\E(X \cdot I_{\{X\geq 2\}}) < np$ for $k > M$, (\ref{eq:ests1rm}) is less than
$$\begin{aligned}
&2\lambda\left( \sum_{k=3}^ M p(k)\cdot k^2\cdot q(M,\infty)^2 + \sum_{k=M+1}^\infty p(k)\cdot k\cdot q(M,\infty)\right)\\
&\leq C\lambda \left( M^{-2(a-2)}\cdot\sum_{k=3}^M p(k)\cdot k^2 + M^{-(a-2)}\cdot\sum_{k=M+1}^\infty p(k)\cdot k\right)\\
&\leq C\lambda \left(M^{-2(a-2)}\cdot M^{3-a} + M^{-(a-2)}\cdot M^{-(a-2)} \right)\\&\leq C\lambda \left(M^{-3a+7} + M^{-2a+4} \right) \leq C\lambda(\lambda^{6a - 14} + \lambda^{4a - 8}) < \rho_a(\lambda)
\end{aligned}$$
when $\lambda$ is small, since $6a-13, 4a-7 > 2a-3$ when $a > 2\frac{1}{2}$.\medskip\\

In order to bound the probabilities of $B^4_4$ and $B^4_5$, we will need the following result, whose proof is omitted.
\begin{lemma}\label{lem:compTrees0}
The degrees of the vertices of $\T$ under $\Q_{p,q}(\;\cdot\;\big|\;\deg(o) \leq M,\;|S_{R,M}^{1}| = 1\;)$ are distributed as follows:\\
$(i.)$\ First, $\deg(o)$ is chosen with distribution
\begin{flalign}\label{eq:lawTrunc}&k \in [0,M] \mapsto \frac{\frac{p(k)}{p[0,M]}\cdot\Q_{p,q}(|S^1_{R,M}|=1 \;\big|\; \deg(o)=k)}{\sum_{w=1}^M\;\frac{p(w)}{p[0, M]}\cdot \Q_{p,q}(|S^1_{R,M}|=1 \;\big|\; \deg(o)=w)}.\end{flalign}
$(ii.)$\ Given the choice of $\deg(o)$, the degrees of $o^*$ and the remaining neighbours of $o$ are chosen independently: $\deg(o^*)$ with law
\begin{equation}\label{eq:lawTrunc1}k \in (M, \infty) \mapsto \left(q(M,\infty)\right)^{-1}q(k) \end{equation}
and the remaining degrees with law
\begin{equation}\label{eq:lawTrunc2}k \in [0,M] \mapsto \left(q[0,M]\right)^{-1}q(k). \end{equation}
$(iii.)$\ All other vertices in the tree have degrees chosen independently with distribution $q$.
\end{lemma}

\noindent \textbf{Remark.} The distribution in (\ref{eq:lawTrunc}) is equal to
\begin{eqnarray}&&k \mapsto \frac{p(k) \cdot k \cdot q(M,\infty)\cdot \left(q[0,M]\right)^{k-1}}{\sum_{w=1}^M p(w) \cdot w \cdot q(M, \infty) \cdot \left(q[0,M] \right)^{w-1}}\nonumber\\&&\qquad \qquad=\left(\sum_{w=1}^M p(w)\cdot w \cdot \left(q[0,M]\right)^{w-1}\right)^{-1} p(k)\cdot k\cdot \left(q[0,M]\right)^{k-1};\nonumber\end{eqnarray}
hence, it is stochastically dominated by $q$. Obviously, the distribution in (\ref{eq:lawTrunc2}) is also dominated by $q$.

\noindent \textbf{4) Event $B^4_4$.} 
\begin{equation} \Q_{p,q}^\lambda(B^4_4) \leq \Q_{p,q}(|S^1_{R,M}|>0) \cdot \Q_{p,q}^\lambda\left(\begin{array}{c}\exists y \in \overline \T_{R,M},\; 0<s<t:\\(o,0) \lra (y,s) \lra (o,t) \text{ inside } \overline \T_{R,M} \end{array}\left| \begin{array}{c}\deg(o) \leq M,\\|S^1_{R,M}| = 1 \end{array}\right.\right).\nonumber\end{equation}
The first probability on the right-hand side is less than $\sum_{k=3}^\infty p(k)\cdot k \cdot q(M,\infty) \leq C\lambda^{2(a-2)}$. By Lemma \ref{lembound}$(v.)$, the second probability is less than
\begin{equation}\label{eq:B24red}\sum_{i=1}^R \lambda^{2i} \cdot \E_{\Q_{p,q}}\left(|\{x \in \overline \T_{R,M}: d(o,x) = i\}|\;\big|\; \deg(o) \leq M,\; |S^1_{R,M}| = 1 \right)  \end{equation}
By Lemma \ref{lem:compTrees0} and the remark that follows it, this conditional expectation is bounded by 
$$\E_{\Q_q}\left(\left|\left\{x \in \overline \T_{R,M}:d(o,x) = i\right\} \right|\right) \leq \left\{\begin{array}{ll}(C_0M^ {3-a})^{i}&\text{if } 2\frac{1}{2} < a < 3;\medskip\\(C_0\log M)^i &\text{if } a = 3;\medskip\\ \nu^i &\text{if } a > 3. \end{array}\right.$$
It is then easy to check that the sum in (\ref{eq:B24red}) is less than $\lambda^{1+\delta}$ for some $\delta > 0$. In conclusion, $\Q_{p,q}^\lambda(B^4_4) \leq \lambda^{2(a-2)+1+\delta} < \rho_a(\lambda)$ when $\lambda$ is small.\medskip\\

\noindent \textbf{5) Event $B^4_5$.}
This is the bound that requires most effort. We start with
\begin{flalign}\nonumber&\Q_{p,q}^\lambda(B^4_5)& \\&\leq C\lambda^{2(a-2)} \cdot \frac{\lambda}{1+\lambda} \cdot \Q_{p,q}^\lambda\left(\;(o^*, t^*) \;\lra \;\T\times [t, \infty) \; \forall t > t^*\;\left|\;\deg(o) \leq M,\; |S_{R,M}^{1}| = 1,\; t^* < \inf D_o\right.\right)&\nonumber\\
\label{eq:B25red}&=C\lambda^{2(a-2)} \cdot \frac{\lambda}{1+\lambda} \cdot \Q_{p,q}^\lambda\left(\;(o^*, 0) \;\lra \;\T\times [t, \infty) \; \forall t > 0\;\left|\;\deg(o) \leq M,\; |S_{R,M}^{1}| = 1\right.\right).&\end{flalign}
In order to deal with the conditioning in (\ref{eq:B25red}), we need the following, which is a consequence of Lemma \ref{lem:compTrees0} and the remark that follows it.

\begin{lemma}\label{lem:compTrees}
Let $\hat \T$ be the random rooted tree obtained by\medskip \\
$\bullet$\ sampling $\T$ under law $\Q_{p,q}(\;\cdot\;|\deg(o)\leq M,\; |S^1_{R,M}|= 1)$;\medskip\\
$\bullet$ repositioning the root at $o^*$, the unique vertex in $S_{R,M}^{1}$.\medskip\\
Then, $\hat \T$ is stochastically dominated by the distribution $\Q_{q}(\;\cdot\;|\deg(o) > M)$.
\end{lemma}
As a consequence of Lemma \ref{lem:compTrees} and attractiveness of the contact process, we get
$$\begin{aligned}&\Q_{p,q}^\lambda\left((o^*, 0) \;\lra \;\T\times \{t\} \; \forall t > 0\;\left|\;\deg(o) \leq M,\; |S_{R,M}^{1}| = 1\right.\right)\leq \Q_q^\lambda\left(\xi^{o}_t \neq \varnothing \; \forall t\;|\; \deg(o) > M\right).\end{aligned}$$
Using this in (\ref{eq:B25red}), we get
\begin{equation}\Q_{p,q}^\lambda(B^4_5) \leq C\lambda^{1+2(a-2)}\cdot \Q_q^\lambda\left(\xi^{o}_t \neq \varnothing \; \forall t\;|\; \deg(o) > M\right).\label{eq:redTom0}\end{equation}
In treating the last term of (\ref{eq:redTom0}), we will obtain the logarithmic term in the definition of $\rho_a(\lambda)$. This is encompassed in the following proposition.
\begin{proposition}
There exists $C > 0$ such that
\begin{equation}\label{eq:redTom}\Q_q^\lambda\left(\xi^{o}_t \neq \varnothing \; \forall t\;|\; \deg(o) > M\right) \leq \left\{\begin{array}{ll}C\log^{-(a-2)}\left(\frac{1}{\lambda}\right) &\text{if } 2\frac{1}{2} < a \leq 3;\smallskip\\C\log^{-2(a-2)}\left(\frac{1}{\lambda}\right) &\text{if } a > 3. \end{array}\right.\end{equation}
\label{prop:redTom}\end{proposition}
Define 
$$M' = \left\{\begin{array}{ll}\lceil \epsilon_1 \frac{1}{\lambda^2}\log\left(\frac{1}{\lambda}\right)\rceil &\text{if } 2\frac{1}{2} < a \leq 3;\medskip\\\lceil \epsilon_2 \frac{1}{\lambda^2}\log^2\left(\frac{1}{\lambda}\right)\rceil &\text{if } a > 3, \end{array}\right.$$
where $\epsilon_1,\; \epsilon_2$ are constants to be chosen later, depending on $a$ but not on $\lambda$. Our approach to prove Proposition \ref{prop:redTom} starts with the following:
$$\begin{aligned}
&\Q_q^\lambda\left(\xi^{o}_t \neq \varnothing \; \forall t \;|\; \deg(o) > M\right) \\&\quad= \sum_{m = \lceil M\rceil}^\infty \Q_q^\lambda\left(\xi^{o}_t \neq \varnothing \; \forall t \;|\; \deg(o) = m\right)\cdot \Q_q\left(\deg(o) = m\;|\; \deg(o) > M\right)\\
&\quad \leq  \Q_q^\lambda\left(\xi^{o}_t \neq \varnothing \; \forall t \;|\; \deg(o) =  M'\right)+\Q_q\left(\deg(o) > M'\;|\: \deg(o)> M\right) \\
&\quad \leq \Q_q^\lambda\left(\xi^{o}_t \neq \varnothing \; \forall t \;|\; \deg(o) = M'\right) + \frac{q[M', \infty)}{q[M, \infty)}.
\end{aligned}$$
Now, by (\ref{c0a-2}), the term $\frac{q[M', \infty)}{q[M, \infty)}$ is bounded from above by the expression in the right-hand side of (\ref{eq:redTom}), for some $C > 0$. Proposition \ref{prop:redTom} will thus follow from
\begin{lemma}\label{lem:redTom}
If $a > 2\frac{1}{2}$, then there exists $\delta > 0$ such that
$$\Q_q^\lambda\left(\xi^{o}_t \neq \varnothing \;\forall t\;|\: \deg(o) = M'\right) < \lambda^\delta.$$
\end{lemma}
In the next two subsections, we prove Lemma \ref{lem:redTom} separately for the cases $2\frac{1}{2} < a \leq 3$ and $a > 3$.

\subsection{Completion of proof for $2\frac{1}{2} < a \leq 3$}
\label{ss:ca>212}
In this subsection and the next, we will consider the probability measure $\Q_{q}^\lambda(\;\cdot\;|\deg(o) = M')$, so $\T$ will be a tree with root degree equal to $M'$. Here we will give the proof in detail for the case $2\frac{1}{2} < a < 3$; the case $a = 3$ is treated similarly and we will omit it for brevity.

Let $\epsilon_1' = \frac{2a-5}{2}$ and $\epsilon_1 = \frac{\epsilon_1'}{64}$; this is the constant that appears in the definition of $M'$. Also let $L_1 = \lambda^{-\epsilon_1'/2}$ and fix an integer $R'$ large enough that $(2a-5)(R'-1) - 1 > 2a-5$. We will be particularly interested in the contact process on $B_{\T}(o, R')$ in the time interval $[0, L_1]$.

We will need the quantities
$$\begin{aligned}
&\upphi(\T) = \sum_{i=1}^{R'} (2\lambda)^i\cdot|S^i_{R',M}(\T)|;\\
&\uppsi(\T) = \sum_{i=2}^{R'} (2\lambda)^{2i} \cdot |\{x \in \overline \T_{R',M}: d(x, o) = i\}|.
\end{aligned}$$
We define two environment events, which are simply
$$B^5_1 = \left\{\upphi(\T) > \lambda^{\epsilon_1'}\right\};\qquad B^5_2 = \left\{\uppsi(\T) >  \lambda^{\epsilon_1'}\right\},$$
and then define three events involving the contact process:
$$\begin{aligned}
&B^5_3 = (B^5_1 \cup B^5_2)^c \cap \left\{\{o\} \times [0,L_1] \;\lra\;\left(\mathop{\cup}_{i=1}^R S_{R',M}^{i} \right) \times \R_+\right\};\\
&B^5_4 = (B^5_1 \cup B^5_2)^c \cap \left\{\exists z \in \overline \T_{R',M},\;s>0: \{o\}\times [0,L_1] \;\lra\;(z,s)\;\lra\; \{o\}\times [s, \infty) \text{ inside }\overline \T_{R',M} \right\};\\
&B^5_5 = \left\{B(o,1) \times \{0\} \;\lra\; B(0,1) \times \{L_1\} \text{ inside } B(o,1)\right\}.
\end{aligned}$$
We claim that $\left\{\xi^o_t \neq \varnothing \;\forall t\right\} \subset \cup_{i=1}^5 B^5_i$. To show this, it suffices to show that, if an infection path $t \mapsto \gamma(t) \in \T$ with $\gamma(0) = o$ ever reaches any point of $\cup_{i=1}^{R'} S^i_{R',M}$, then one of the events must occur. Let $t^* = \inf\{t: \gamma(t) \in \cup_{i=1}^{R'} S^i_{R',M}\}$ and $t^{**} = \sup\{t \leq t^*: \gamma(t) = o\}$. If $t^{**} \leq L_1$, then $B^5_3$ occurs. If $t^{**} > L_1$ and $\gamma(t) \in B(o,1)$ for all $t \in [0, t^{**}]$, then $B^5_5$ occurs. Otherwise, $B^5_4$ occurs.

We now want to show that the probability of each of the five events is less than $\lambda^\delta$ when $\lambda$ is small, for some $\delta > 0$.\medskip\\

\noindent \textbf{1) Event $B^5_1$.} Bounding as in (\ref{eq:expSRM1}), we have
\begin{eqnarray}
&&\E_{\Q_q}\left(\upphi(\T)\;|\; \deg(o) = M' \right) \nonumber\\&&\leq \sum_{i=1}^{R'-1}(2\lambda)^i \cdot M'\cdot (C_0M^{3-a})^{i-1} \cdot C_0M^{-(a-2)} + (2\lambda)^{R'}\cdot M' \cdot (C_0M^{3-a})^{R'-1}.\label{eq:lastAux}
\end{eqnarray}
The first term in (\ref{eq:lastAux}) is less than
$$2\lambda \cdot \epsilon_1\;\frac{1}{\lambda^2}\;\log\left(\frac{1}{\lambda}\right) \cdot C_0 (8\lambda^2)^{a-2}\cdot \sum_{i=1}^{R'-1}\left(C_0\left(\frac{1}{8\lambda^2}\right)^{3-a} \cdot 2\lambda\right)^{i-1} \leq C\lambda^{2a-5}\cdot \log\left(\frac{1}{\lambda}\right).$$
The second term in (\ref{eq:lastAux}) is less than
$$2\lambda \cdot \epsilon_1\;\frac{1}{\lambda^2}\;\log\left( \frac{1}{\lambda}\right)\cdot \left(C_0 \left( \frac{1}{8\lambda^2}\right)^{3-a} \cdot 2\lambda\right)^{R'-1}= 2\epsilon_1 \cdot \frac{1}{\lambda}\log \left(\frac{1}{\lambda}\right)\cdot \left(C_0\cdot \frac{1}{8^{3-a}}\cdot \lambda^{2a-5} \right)^{R'-1}$$
and this is also less than $C\lambda^{2a-5}\log(1/\lambda)$ by the choice of $R'$. This shows that
\begin{equation}
\nonumber\E_{\Q_q}\left(\upphi(\T)\;|\;\deg(o) = M' \right) \leq C \lambda^{2a-5}\cdot \log(1/\lambda)
\end{equation}
Thus, by the Markov inequality,
$$\Q_{q}\left(B^5_1\;|\;\deg(o) = M'\right) \leq \frac{C\lambda^{2a-5}\log(1/\lambda)}{\lambda^{(2a-5)/2}} < \lambda^{(2a-5)/4}.$$\\

\noindent \textbf{2) Event $B^5_2$.} Bounding as in (\ref{eq:expSRM2}),
$$\begin{aligned}
&\E_{\Q_q}\left(\uppsi(\T)\;|\;\deg(o) = M'\right) \leq M'\cdot \sum_{i=2}^{R'} (2\lambda)^{2i}\cdot (C_0M^{3-a})^{i-1}\\
&\leq \epsilon_1 \; \frac{1}{\lambda^2} \log\left(\frac{1}{\lambda}\right)\cdot \lambda^{3} \cdot\sum_{i=2}^{R'}(2\lambda)^{2i-3}\cdot\left(C_0 \left(\frac{1}{8\lambda^2}\right)^{3-a} \right)^{i-1} < C\lambda  \log \left(\frac{1}{\lambda}\right),
\end{aligned}$$
since the exponent of $\lambda$ inside the sum, $2i - 3 - 2(3-a)(i-1) = (2a-4)i + 3 -2a$, is positive when $i \geq 2$. The desired bound now follows from the Markov inequality as above.\medskip\\

\noindent \textbf{3) Event $B^5_3$.} For $x \in \T,\;x \neq o$, let $s(x)$ denote the neighbour of $o$ in the geodesic from $o$ to $x$, and let $\T(x)$ be the subtree of $\T$ with vertex set $$\{o\} \cup \{y\in \T: \text{the geodesic from $o$ to $y$ contains $s(x)$}\}$$ and edge set $\{\{z,w\}:z \sim w \text{ in } \T,\;z, w \in \T(x)\}$. 

For $B^5_3$ to occur, there must exist $x \in \cup_{i=1}^{R'} S^i_{R',M}$ so that $\{o\} \times [0, L_1] \;\lra\;\{x\}\times \R_+$ inside $\T(x) \cap \overline \T_{R',M}$. For a fixed $x$, the probability of such a path is less than $(L_1+1)\cdot(2\lambda)^{d(o,x)}$ by Lemma \ref{lembound}$(iv.)$, since $\T(x) \cap \overline \T_{R',M}$ is a tree in which all degrees are bounded by $M$. Summing over all $x$, this yields
$$P_{\T}^\lambda\left(\{o\} \times [0,L_1] \; \lra \;\mathop{\cup_{i=1}^{R'}}S^i_{R',M} \times \R_+ \right) \leq (L_1+1)\cdot \upphi(\T).$$
If $B^5_1$ does not occur, then the right-hand side is less than $(L_1+1)\cdot \lambda^{\epsilon'_1} = (\lambda^{-\epsilon_1'/2}+ 1)\cdot \lambda^{\epsilon_1'}$. Thus,
$$\Q_q^\lambda(B^5_3\;|\;(B^5_1)^c) < \lambda^{\epsilon_1'/2}+ \lambda^{\epsilon_1'}.$$\\

\noindent \textbf{4) Event $B^5_4$.} This is treated similarly to the previous event; here we use Lemma  \ref{lembound}$(vi)$ to conclude that
$$\Q_q^\lambda(B^5_4\;|\;(B^5_2)^c) < (L_1 + 1)\cdot \lambda^{\epsilon_1'} = \lambda^{\epsilon_1'/2}+ \lambda^{\epsilon_1'}.$$\\

\noindent \textbf{5) Event $B^5_5$.} For $i \leq 0$, let
$$E_i = \left\{B(o, 1) \times \{i\cdot 3\log(1/\lambda)\} \;\nleftrightarrow B(o,1) \times \{(i+1)\cdot 3\log(1/\lambda)\}\;\right\}.$$
These events are independent and, by Lemma \ref{lem:extStar},
\begin{equation}\Q^\lambda_q(E_i\;|\;\deg(o) = M') \geq (1/4)e^{-16\lambda^2M'} = (1/4)\lambda^{16\epsilon_1} = (1/4)\lambda^{\epsilon_1'/4}. \label{eq:compareStarl1}\end{equation}
If $B(o,1) \times \{0\} \; \lra \; B(0,1) \times \{L_1\}$, then $E_i$ cannot occur for
\begin{equation} 0 \leq i \leq \lfloor L_1/(3\log(1/\lambda)) \rfloor = \lfloor \lambda^{-\epsilon_1'/2}/(3\log(1/\lambda)) \rfloor.
\label{eq:compareStarl2}\end{equation}
Comparing (\ref{eq:compareStarl1}) and (\ref{eq:compareStarl2}), it is easy to see that $\Q_q^\lambda(B^5_5)$ is smaller than any power of $\lambda$ as $\lambda \to 0$.

\subsection{Completion of proof for $a > 3$}
\label{ss:ca>3}
Fix $\epsilon_2' > 0$ with $\epsilon_2' < \min\left((4\log\nu)^{-1},\; a-3\right)$ and set $\epsilon_2 = \frac{\epsilon_2'}{18}$; this is the constant that appears in the definition of $M'$. Also define $R' = \lceil\epsilon_2' \log \frac{1}{\lambda} \rceil$ and $L_2=\lambda^{-17\epsilon_2\log\frac{1}{\lambda}}$. We will be particularly interested in the contact process on $B_\T(o, R')$ in the time interval $[0, L_2]$.

This time, our environment events correspond to violations of the properties required for Lemma \ref{lem:extTree} to be applied:
$$\begin{aligned}
&B^6_1 = \left\{|B(o, R')| > \lambda^{-3}\right\};\\
&B^6_2 = \{\exists x \in \T\backslash\{o\}: d(o,x) \leq R',\;\deg(x) > M\}.\end{aligned}$$
The first event involving the contact process is the existence of an infection path starting on $\{o\}\times[0,L_2]$ and reaching vertices at distance more than $R'$ from the root,
$$B^6_3 = (B^6_1 \cup B^6_2)^c \cap \left\{\{o\}\times \left[0, L_2\right] \;\lra\; B(o, R')^c \times \R_+\right\},$$
The second event is the infection surviving up to time $L_2$ without leaving the ball $B(o, R')$,
$$B^6_4 = (B^6_1 \cup B^6_2)^c \cap \left\{B(o, R') \times \{0\} \;\lra \; B(o, R') \times \{L_2\} \text{ inside } B(o, R')\right\}.
$$
Again we have $\left\{\xi^o_t \neq \varnothing \; \forall t\right\} \subset \cup_{i=1}^4 B^6_i$. We proceed to show that each of these event has probability smaller than $\lambda^\delta$, for some $\delta > 0$.\medskip\\

\noindent \textbf{1) Event $B^6_1$.} Using Markov's inequality,
$$\begin{aligned}\Q_q(B^6_1\;|\;\deg(o) = M')&\leq \lambda^3\cdot \E_{\Q_q}\left(|B(o, R')|\;\big|\;\deg(o) = M'\right)\\
&\leq \lambda^3\cdot R'\cdot \E_{\Q_q}\left(|\{x \in \T: d(o,x) = R'\}|\; \big|\;\deg(o) = M'\right)\\
&\leq \lambda^3 \cdot R' \cdot M' \cdot \nu^{R'}\\
&\leq \lambda^3 \cdot 2\epsilon_2'\log \frac{1}{\lambda} \cdot 2\epsilon_2 \frac{1}{\lambda^2}\log^2\left(\frac{1}{\lambda}\right) \cdot \nu^{\frac{1}{4\log \nu}\log \frac{1}{\lambda}} < \lambda^{1/2}
\end{aligned}$$
when $\lambda$ is small.\medskip\\

\noindent \textbf{2) Event $B^6_2$.}
$$\begin{aligned}\Q_q(B^6_2\;|\;\deg(o)=M') &\leq \sum_{i=1}^{R'} \epsilon_2 \frac{1}{\lambda^2}\log\frac{1}{\lambda}\cdot \nu^{i-1}\cdot q(M,\infty) \\&\leq \lambda^{2(a-2)-2}\cdot \log \frac{1}{\lambda} \cdot \nu^{R'+1} < \lambda^{2a-6-2\epsilon'};\end{aligned}$$
by the choice of $\epsilon_2'$, $2a-6-2\epsilon' > 0$, so we are done.\medskip\\

\noindent \textbf{3) Event $B^6_3$.} Assume $(B^6_1 \cup B^6_2)^c$ occurs, so that $\T$ is such that $|B_\T(o, R')| \leq \lambda^{-3}$ and, with the exception of the root $o$, the degrees of all vertices in $B_\T(o, R')$ are less than $M$. Recall from the previous subsection (in the treatment of the event $B^5_3$) the definition of $\T(x)$ for a vertex $x \neq o$. Note that presently, for any $x \in B_\T(o, R')$, $\T(x)$ is a tree in which all degrees are bounded by $M$.

If $\{o\} \times [0,L_2] \;\lra\;B_\T(o, R')^c \times \R_+$, then there must exist $x$ with $d(o,x) = R'$ so that $\{o\} \times [0, L_2] \;\lra\;\{x\} \times \R_+$ inside $\T(x)$. For a fixed $x$, the probability of this is bounded by $(L_2 + 1)\cdot (2\lambda)^{R'}$, by Lemma \ref{lembound}$(iv.)$, so
$$P^\lambda_\T\left(\{o\} \times [0,L_2] \;\lra\; B(o, R')^c \times \R_+\right) \leq |\{x: d(o,x) = R'\}|\cdot (L_2+1)\cdot(2\lambda)^{R'} \leq \lambda^{-3}\cdot (L_2+1)\cdot(2\lambda)^{R'},$$
so that
$$\Q_q^ \lambda(B^6_3\;|\;(B^6_1 \cup B^6_2)^c,\;\deg(o) = M') \leq \lambda^{-3}\cdot \left(\lambda^{-17\epsilon_2\log\frac{1}{\lambda}}+1\right)\cdot (2\lambda)^ {\epsilon_2'\log\frac{1}{\lambda}},$$
so, using the fact that $\epsilon_2 = \frac{\epsilon_2'}{18},$ we are done.\medskip\\

\noindent \textbf{4) Event $B^6_4$.} Again assume that $(B^6_1 \cup B^6_2)^c$ occurs. For $i \geq 0$, let
$$F_i = \left\{B(o, R')\times \left\{i\cdot 100\log(1/\lambda)\right\}\nleftrightarrow B(o, R') \times \left\{(i+1)\cdot 100\log(1/\lambda)\right\}\right\}.$$
These events are independent and, by Lemma \ref{lem:extTree}, 
\begin{equation} P^\lambda_\T(F_i) \geq (1/8)e^{-16\lambda^2M'} = (1/8)\lambda^{16\epsilon_2\log(1/\lambda)}\label{eq:lastComp1}.\end{equation} If $B(o, R') \times \{0\} \;\lra\; B(o, R') \times \{L_2\}$, then $F_i$ cannot occur for 
\begin{equation} 0 \leq i \leq \lfloor L_2/(100\log(1/\lambda)) \rfloor = \lfloor \lambda^{-17\epsilon_2\log(1/\lambda)}/(100\log(1/\lambda)) \rfloor.\label{eq:lastComp2}\end{equation}
Comparing (\ref{eq:lastComp1}) and (\ref{eq:lastComp2}), it is easy to see that $\Q_q^\lambda(B^6_4\;|\;(B^6_1\cup B^6_2)^c)$ is smaller than any power of $\lambda$ as $\lambda \to 0$.

\section{Appendix: Proof of Theorem \ref{thm:reduc}}\label{s:appendix}
Recall the definition of $\mathcal{M}(x, R, K)$ and $\mathcal{N}(x, R, K)$ in (\ref{eq:defM}) and (\ref{eq:defN}). We will also need
$$\mathcal{L}(x,R) = \left\{(x,0) \;\lra\; B(x, R)^c \times \R_+ \right\}.$$

\begin{lemma}
\label{lem:aphelp}
For any $\epsilon > 0$ and $\lambda > 0$ there exists $K_0 > 0$ such that, for any $K > K_0$,
\begin{equation}\Q^\lambda_{q}\left(\left.{\mathop\cap_{i=K+1}^{\infty}} \;\mathcal{N}\left(o, a\log_2(i), i\right)\; \right|\;\deg(o) = K\right) > 1 - \epsilon.\label{eq:auxappend}\end{equation}
\end{lemma}
Clearly, it is enough to prove the above for $\lambda$ small enough. This is done using Lemma \ref{basic}$(iii.)$ and Lemma \ref{lem:infNei}; since the proof is essentially a repetition of the ideas of Subsection \ref{ss:lower}, we omit it.

The point of the following lemma is approximating the event $\{\xi^o_t \neq \varnothing \;\forall t\}$ on the infinite tree $\T$ by events involving the contact process on a finite ball around the root, $B(o, R)$.
\begin{lemma}
For any $\epsilon > 0,\; \lambda > 0$ and $R_0 > 0$, there exists $R > R_0$ such that
\begin{equation}\Q_{p,q}^\lambda\left(\mathcal{L}(o,R)\right) - \epsilon < \Q_{p,q}^\lambda\left(\xi^o_t \neq \varnothing \; \forall t \right) < \Q_{p,q}^\lambda\left(\mathcal{N}(o,R,R^2)\right) + \epsilon.\nonumber\end{equation}\label{lem:compFinEv}
\end{lemma}
\begin{proof} Fix $\epsilon,\;\lambda$ and $R_0$. The existence of $R$ such that the first inequality is satisfied is a direct consequence of 
$\left\{\xi^o_t \neq \varnothing \; \forall t \right\} = \cap_{r=1}^\infty\; \mathcal{L}(o,R).$

Let us now deal with the second inequality. For the process $(\xi^o_t)_{t\geq 0}$, let $\sigma_i$ be the first time the infection reaches a vertex at distance $i$ from the root, and $X_i$ the vertex that becomes infected at this time. Define $N_k = \inf\{i: \deg(X_i) > k\}$. Since
$$\begin{aligned}
&\lim_{k \to \infty} \Q^\lambda_{p,q}\left(\left.N_k < \infty \;\right|\;\exists t: \xi^o_t = \varnothing \right) = 0 \text{ and }\\
&\lim_{r\to\infty}\Q^\lambda_{p,q}\left(\left.N_k < \infty,\;d(o,X_{N_k}) < r \;\right|\;\xi^o_t \neq \varnothing \;\forall t\right) = 1 \text{ for all } k,
\end{aligned}$$
we can choose $r_0, k_0$ such that 
$$\Q_{p,q}^\lambda(N_{k_0} < \infty,\; d(o, X_{N_{k_0}}) < r_0) > \Q_{p,q}^\lambda\left(\xi^o_t \neq \varnothing \; \forall t \right) - \epsilon.$$
Also assume $k_0$ is large enough that (\ref{eq:auxappend}) is satisfied when $K = k_0 - 1$.

Choose $k_1$ large enough that $\left(r_0 + a \log_2 k_1\right)^2 < k_1$ and $r_0 + a \log_2 k_1 > R_0$. We define the event $\mathcal{N}'\left(X_{N_{k_0}}, a\log_2 k_1, k_1 \right)$ as the event $\mathcal{N}\left(X_{N_{k_0}}, a\log_2 k_1, k_1 \right)$ with time shifted so that $\sigma_{N_{k_0}}$ becomes the time origin (so that the infection starts at the space-time point $(X_{N_{k_0}},\;\sigma_{N_{k_0}})$). By the definition of $N_{k_0}$ and the choice of $k_0$,
$$\begin{aligned}&\Q_{p,q}^\lambda\left(\left.\mathcal{N}'\left(X_{N_{k_0}}, a\log_2 k_1, k_1 \right) \right| \;N_{k_0} < \infty,\;d(o, X_{N_{k_0}}) < r_0  \right)\medskip\\
&\geq \Q^\lambda_q\left( \;\mathcal{N}(o, a\log_2 k_1, k_1 )\;|\;\deg(o) = k_0 - 1 \right) > 1 - \epsilon.\end{aligned}$$

We have thus shown that, with probability larger than $(1-\epsilon)\left(\Q^\lambda_{p,q}\left(\xi^o_t \neq \varnothing \forall t\right) -\epsilon \right)$, the infection reaches a site $X_{N_{k_0}}$ of degree larger than $k_0$ and distance less than $r_0$ from the root and then, reaches a site $y$ of degree larger than $k_1$ and distance less than $a\log_2 k_1$ from $X_{N_{k_0}}$. All this occurs through infection paths through vertices whose distance from the root is never more than $R := r_0 + a \log_2 k_1$, so that $R > R_0$ and $R^2 < k_1$ as required. Since $\epsilon$ is arbitrary, the proof is complete.
\end{proof}

\begin{lemma}
\label{lem:NGn}
For any $\epsilon > 0,\; \lambda > 0$ and $(t_n)$ with $\log t_n = o(n)$, there exists $R>0$ such that
$$\liminf_{n \to \infty}\; \P^\lambda_{p, n}\left(\xi^{v_1}_{t_n} \neq \varnothing \;|\;\mathcal{N}(v_1, R, R^2)\right) > 1 - \epsilon.$$
\end{lemma}
Since the proof of this lemma requires several preliminary results, we will postpone it. With Lemmas \ref{lem:compFinEv} and \ref{lem:NGn} at hand, we are ready for our main proof.

\bprthm{thm:reduc}
Fix $\lambda >0,\;\epsilon >0$ and $(t_n)$ with $t_n \to \infty$ and $\log t_n = o(n)$. We will write $\upgamma =  \Q_{p,q}^\lambda \left(\xi^o_t \neq \varnothing \;\forall t\right)$.

By Lemmas \ref{lem:compFinEv} and \ref{lem:NGn}, we can choose $R > 0$ such that
\begin{equation} \begin{array}{l} 
\Q_{p,q}^\lambda\left(\mathcal{L}(o,R)\right) - \epsilon < \upgamma < \Q_{p,q}^\lambda\left(\mathcal{N}(o,R,R^2)\right) + \epsilon;\medskip\\
{\displaystyle \limsup_{n \to \infty}}\;\P_{p,n}^\lambda \left(\xi^{v_1}_{t_n} = \varnothing \;|\;\mathcal{N}(v_1, R, R^2) \right) < \epsilon^2.
\end{array}
\label{eq:auxCompLem}
\end{equation}
For the contact process with parameter $\lambda$ on $G_n$, define:
$$X_{n,i} = I_{\left\{\xi^{v_i}_{t_n} \neq \varnothing \right\}},\qquad \overline X_{n,i} = I_{\mathcal{L}(v_i, R)},\qquad  Y_{n,i} = I_{\mathcal{L}(v_i, R)^c\;\cap \;\left\{\xi^{v_i}_{t_n} \neq \varnothing\right\}},\qquad 1\leq i \leq n.$$
Under $\P_{p,n}^\lambda$, $(X_{n,1},\ldots, X_{n,n}),\;(\overline X_{n,1},\ldots, \overline X_{n,n})$ and $(Y_{n,1},\ldots, Y_{n,n})$ are exchangeable random vectors with
$X_{n,i} \leq \overline X_{n,i} + Y_{n,i}$ for each $i$ and, by Proposition \ref{lem:kGW},
\begin{equation}\label{eq:XXY2}
\begin{array}{l}{\displaystyle \lim_{n \to \infty}}\;\P_{p,n}^\lambda\left(\overline X_{n,1} = 1\right) = \Q_{p,q}^\lambda(\mathcal{L}(o,R)) < \upgamma + \epsilon;\medskip\\ {\displaystyle \lim_{n \to \infty}}\;\P_{p,n}^\lambda\left(Y_{n,1} = 1\right) = 0;\medskip\\{\displaystyle \lim_{n \to \infty}}\;\mathsf{Cov}(\overline X_{n,1}, \overline X_{n,2}) = 0.\end{array}
\end{equation}
Using duality for the contact process, we have
$$\begin{aligned}\P_{p,n}^\lambda\left(|\xi^{V_n}_{t_n} | > (\upgamma + 3\epsilon)n \right)&= \P_{p,n}^\lambda\left(\sum_{i=1}^n X_{n,i} > (\upgamma + 3\epsilon)n \right)\\
&\leq\P_{p,n}^\lambda\left(\sum_{i=1}^n \overline X_{n,i} > (\upgamma + 2\epsilon)n \right) + \P_{p,n}^\lambda\left(\sum_{i=1}^n Y_{n,i} > \epsilon n \right) \stackrel{n \to \infty}{\xrightarrow{\hspace*{0.8cm}}}0
\end{aligned}$$
by (\ref{eq:XXY2}).

We now define
$$\overline Z_{n,i} = I_{\mathcal{N}(v_i, R, R^2)},\qquad W_{n,i} = I_{\mathcal{N}(v_i, R, R^2) \; \cap \;\left\{\xi^{v_i}_{t_n} = \varnothing\right\}}, \qquad 1 \leq i \leq n.$$
Again, we get exchangeable random vectors and $X_{n,i} \geq \overline Z_{n,i} - W_{n,i}$. By Proposition \ref{lem:kGW} and (\ref{eq:auxCompLem}),
\begin{equation}
\label{eq:ZZW2}\begin{array}{l}
{\displaystyle \lim_{n \to \infty}}\P_{p,n}^\lambda\left(\overline Z_{n,1} = 1 \right) = \Q_{p,q}^\lambda\left(\mathcal{N}(o, R, R^2)\right) > \upgamma - \epsilon,\medskip\\
{\displaystyle \limsup_{n\to\infty}}\;\P_{p,n}^\lambda\left(W_{n,1} = 1\right) < \epsilon^2,\medskip\\
{\displaystyle \lim_{n \to \infty}}\;\mathsf{Cov}(\overline Z_{n,1}, \overline Z_{n,2}) = 0.
\end{array}
\end{equation}
We then have
$$\begin{aligned}\P_{p,n}^\lambda\left(|\xi^{V_n}_{t_n}| < (\upgamma - 3\epsilon)n\right) &=\P_{p,n}^\lambda\left(\sum_{i=1}^n X_{n,i} < (\upgamma - 3\epsilon)n \right)\\
&\leq \P_{p,n}^\lambda\left(\sum_{i=1}^n \overline Z_{n,i} < (\upgamma - 2\epsilon)n \right) + \P_{p,n}^\lambda\left(\sum_{i=1}^n W_{n,i} > \epsilon n \right).
\end{aligned}$$
The first term in the right-hand side vanishes as $n \to \infty$; by Markov's inequality, the second term is less than
$$\frac{n\E_{n,p}^\lambda\left(W_{n,1}\right)}{\epsilon n} \leq \frac{2\epsilon^2n}{\epsilon n} = 2\epsilon$$
when $n$ is large. Since $\epsilon$ is arbitrary, the proof is now complete.
\eprthm

We now turn to the proof of Lemma \ref{lem:NGn}. Let us first explain our approach. We want the infection started at $v_1$ to survive until time $t_n$. Lemma \ref{lem:SurCD} below guarantees that, to this end, it is enough to show that the infection reaches a vertex of degree $n^\delta$. Lemmas \ref{lem:echainb} and \ref{lem:echain} show that with high probability, there exists a ``bridge'' of vertices of increasing degree that can take the infection from $v_1$ to a site of degree larger than $n^\delta$. In order to cross this bridge, the infection needs some ``initial strength'', which is provided by the event $\mathcal{N}(v_1, R, R^2)$ in the conditioning in the probability in Lemma \ref{lem:NGn}.

The following result was proved in \cite{CD} for $t_n = e^{n^\beta}$, where $\beta < 1$. Applying the exponential extinction time result of \cite{MMVY}, it is easy to improve this to $t_n$ satisfying $\log(t_n) = o(n)$.
\begin{lemma}
\label{lem:SurCD}
For any $\delta,\; \epsilon,\;\lambda > 0$ and $(t_n)$ with $\log t_n = o(n)$, we have
$$\P_{p,n}\left(\min_{v \in V_n:\; \deg(v) \geq n^\delta} \;P_{G_n}^\lambda\left(\xi^v_{t_n} \neq \varnothing\right) > 1 - \epsilon \right) \stackrel{n \to \infty}{\xrightarrow{\hspace*{0.8cm}}}1.$$
\end{lemma}
\noindent In words: as $n$ becomes large, the probability of the following converges to 1: the graph $G_n$ is such that, starting the $\lambda$-contact process with a single infection at any site of degree larger than $n^\delta$, with probability larger than $1-\epsilon$ the process will still be active by time $t_n$.

In order to prove the two following lemmas, we describe an alternate, algorithmic construction of the random graph $G_n$. Let $d_1, \ldots,d_n$ be independent with law $p$ and, by adding a half-edge to some vertex if necessary, assume that $\sum_{i=1}^n d_i$ is even. We will match pairs of half-edges, one pair at a time. Let $\mathcal{H}$ denote the set of half-edges. To start, we select a half-edge $h_1$ in any way we want and then choose a half-edge $h_2$ uniformly at random from $\mathcal{H}\backslash\{h_1\}$. We then match $h_1$ and $h_2$ to form an edge. Next, we select a half-edge $h_3$ from $\mathcal{H} \backslash \{h_1, h_2\}$, match it to a half-edge $h_4$ uniformly chosen from $\mathcal{H} \backslash \{h_1, h_2, h_3\}$, and so on, until there are no more half-edges to select. With a moment's reflection, we see that the random graph produced from this procedure is $G_n$.

\begin{lemma}
\label{lem:echainb}
There exists $\kappa = \kappa(a) > 0$ such that, with probability tending to 1 as $n \to \infty$, no cycle is formed when less than $n^\kappa$ matchings of half-edges are made.
\end{lemma}
\begin{proof}
Let $\sigma = \frac{1}{4(a-1)},\; \kappa = \frac{a-2}{9(a-1)}$. Define
$$A = \left\{\sum_{i=1}^n d_i > \frac{n\mu}{2},\; |\{i: d_i > n^\sigma\}| > n^{1-2\sigma(a-1)},\;\sum_{i:\; d_i > n^\sigma} d_i \leq n^{1-\frac{\sigma}{2}(a-2)} \right\}.$$
Using the Law of Large Numbers, (\ref{c0a-1}) and (\ref{c0a-2}), we get $\P_{p,n}(A) \to 1$ as $n \to \infty$. Assume $A$ occurs and we have matched $j$ pairs of half-edges, with $j < n^\kappa$. Let $J$ be the set of vertices associated to half-edges that were matched; we have $|J| \leq 2j < 2n^\kappa < n^{1-2\sigma(a-1)} = \sqrt{n}$ since $\kappa <1/2$. Suppose we now choose a half-edge uniformly at random from the set of half-edges that have not yet been matched. The probability that the chosen half-edge belongs to a vertex that is not in $J$ is
$$\frac{\sum_{i:\;v_i \notin J} d_i}{\sum_{i=1}^n d_i - 2j} = 1 -\frac{\sum_{i:\;v_i \in J}d_i - 2j}{\sum_{i=1}^n d_i - 2j} \geq 1 - \frac{\sum_{i:\;v_i \in J}d_i}{n\mu/4}$$
since $\sum_{i=1}^n d_i > n\mu/2$ and $j << n$. Since $|\{i: d_i > n^\sigma\}| > |J|$, the right-hand side is larger than
$$1-\frac{\sum_{i:\;d_i >n^\sigma}d_i}{n\mu/4} > 1 - \frac{n^{1-\frac{\sigma}{2}(a-2)}}{n\mu/4} = 1 -\frac{4}{\mu n^{\frac{\sigma}{2}(a-2)}}.$$
So the probability of forming a cycle in $\lfloor n^\kappa \rfloor$ matchings is less than
$n^\kappa\cdot \frac{4}{\mu n^{(\sigma/2)(a-2)}} \stackrel{n \to \infty}{\xrightarrow{\hspace*{0.8cm}}} 0$
since $\kappa < \frac{\sigma}{2}(a-2)$.
\end{proof}

\begin{lemma}
\label{lem:echain}
There exists $\delta = \delta(a) > 0$ such that the following holds. For any $\epsilon > 0$, there exists $K_0$ such that, for any $K > K_0$ and $n$ large enough,
$$\P_{p,n}\left({\mathop\cap_{k=K}^{\lceil n^\delta \rceil}}\;\mathcal{M}(x_1, a\log_2 k, k)\right) > 1 - \epsilon.$$
\end{lemma}
\begin{proof}
Let $\kappa$ be as in the above lemma; set $\delta = \frac{\kappa}{a}$ and $N_n = \lceil n^\delta \rceil$. Define the event
$$A'(K) = {\mathop \cap_{k=K}^{N_n}}\left\{\frac{\sum_{i:\; d_i \geq k}\; d_i}{\sum_{i=1}^n d_i} > \frac{1}{k^{a-1}} \right\}.$$
We have
$$\P_{p,n}(A'(K)) \geq 1 - \P_{p,n}\left(\sum_{i=1}^n d_i > 2n\mu\right) - \sum_{k=K}^{N_n} \P_{p,n}\left(\sum_{i:\;d_i \geq k} d_i \leq \frac{2n\mu}{k^{a-1}}\right).$$
For fixed $k$, we have
$$\P_{p,n}\left(\sum_{i:\;d_i \geq k} d_i \leq \frac{2n\mu}{k^{a-1}} \right) \leq \P_{p,n}\left(|\{i: d_i \geq k\}|\leq \frac{2n\mu}{k^a} \right).$$

Now, letting $X \sim \mathsf{Bin}(n, p[k, \infty))$, the probability in the right-hand side is less than 
$$\P\left(X \leq 2n\mu/k^{a}\right) \leq e^{-c\frac{n}{k^{a}}};$$
by (\ref{c0a-1}) and (\ref{mark}). We have thus shown that ${\displaystyle\liminf_{n \to \infty}}\; \P_{p,n}(A'(K)) \geq 1 - \sum_{k=K}^{N_n}e^{-cnk^{-a}} > 1 - \epsilon$ when $K$ is large enough. Fix one such $K$.

We now start matching half-edges; we first match all half-edges incident to $v_1$, then the half-edges incident to the neighbours of $v_1$, and so on. We continue until either a cycle is formed with the edges that we have built (call this a \textit{failed exploration}) or we have revealed more than $n^\kappa$ vertices (a \textit{successful exploration}). By the above lemma, as $n \to \infty$, with high probability we have a successful exploration. We remark that, since all vertices have degree larger than 2, in a successful exploration we reveal at least $2^i$ vertices at distance $i$ from $v_1$, for $0 \leq i \leq \lfloor \log_2 n^\kappa \rfloor$.

Assume $A'(K)$ occurs and let $k \in [K, N_n]$. If at some point in the exploration, $j$ matchings have already been made and no vertex of degree larger than $k$ has been found, then the probability that the next revealed vertex has degree larger than $k$ is larger than $\left(\sum_{i:\;\deg(v_i) \geq k}d_i\right)/\left(\sum_{i=1}^n d_i - 2j\right) \geq k^{-(a-1)}$. Thus,
$$\begin{aligned}&\P_{p,n}\left( \mathcal{M}(v_1, \log_2 k, k)\;\left|\;A'(K) \cap \left\{\text{Successful exploration}\right\}\right. \right)\\
&\geq \P_{p,n}\left(\begin{array}{c}\text{A vertex of degree larger than $k$ is}\\\text{found in $k^a$ steps in the exploration} \end{array}\left|\;A'(K)\cap\left\{\begin{array}{c}\text{Successful}\\\text{exploration} \end{array} \right\} \right.\right)\\
&\geq 1-(1-k^{-(a-1)})^{k^a} \geq 1-e^{-k}.
\end{aligned}$$
This completes the proof.
\end{proof}

\bprlem{lem:NGn}
Fix $\epsilon, \lambda$ and $(t_n)$ as in the statement of the lemma. Since for $R$ large enough, ${\displaystyle\lim_{n \to \infty}}\; \P^\lambda_{p,n}\left(\mathcal{N}(v_1, R, R^2)\right) = \Q_{p,q}^\lambda\left(\mathcal{N}(o, R, R^2)\right) > \frac{1}{2}\;\Q_{p,q}\left(\xi^o_t \neq \varnothing\; \forall t \right) > 0,$
the lemma will follow if we prove that for $R$ large enough,
\begin{equation}\limsup_{n \to \infty} \;\P_{p,n}^\lambda\left(\mathcal{N}(v_1, R, R^2)\cap \left\{\xi^{v_1}_{{t_n}}= \varnothing\right\} \right) < \epsilon. \label{eq:auxSuf1}\end{equation}
Recall that, if $\mathcal{N}(v_1, R, R^2)$ occurs, then there exist $y^*, t^*$ so that $d(v_1, y^*)<R,\;\deg(y^*) > R^2$ and $\frac{|\xi^{v_1}_{t^*}\;\cap\; B(y^*,1)|}{|B(y^*,1)|} > \frac{\min(\lambda, \lambda_0)}{16e}$. So, to prove (\ref{eq:auxSuf1}) it suffices to prove that for $R$ large enough,
\begin{equation}
\liminf_{n \to \infty}\; \P_{p,n}\left(\begin{array}{c}
\text{for all $y^* \in B(v_1, R)$ with $\deg(y^*) > R^2$},\medskip\\\frac{|\xi_{0}\;\cap\; B(y^*,1)|}{|B(y^*,1)|} > \frac{\min(\lambda, \lambda_0)}{16e} \Longrightarrow P^\lambda_{G_n}\left(\xi^{y^*}_{t_n} \neq \varnothing\right)> 1-\epsilon
\end{array}\right) > 1-\epsilon.\label{eq:auxNew}
\end{equation}
Also, it is enough to prove (\ref{eq:auxNew}) under the assumption that $\lambda$ is small enough, so we take $\lambda < \lambda_0$, where $\lambda_0$ is as in Lemma \ref{lem:infNei}.

Fix $\delta > 0$ and $K_0$ corresponding to $\epsilon/2$ in Lemma \ref{lem:echain}. Then take $K \geq K_0$ such that
\begin{equation} 2\sum_{k=K}^\infty e^{-c_1\lambda^2 k} < \frac{\epsilon}{3}, \qquad k > \frac{3}{c_1} \cdot\left(\frac{1}{\lambda} \log \frac{1}{\lambda} \right)\cdot 2a\log_2(k+1) \; \forall k \geq K.\label{eq:pNGn1}\end{equation}
Next, choose $R > 0$ such that
\begin{equation}2e^{-c_1\lambda^2R^2} < \frac{\epsilon}{3},\qquad R^2>\frac{3}{c_1}\cdot \frac{1}{\lambda^2}\log\frac{1}{\lambda}\cdot (R+a\log_2K). \label{eq:pNGn2}\end{equation}
Now define the events for the graph $G_n$:
$$B_1 = \left\{\min_{v \in V_n:\; \deg(v) \geq n^\delta} P^\lambda_{G_n}\left(\xi^v_{t_n} \neq \varnothing\right) > 1-\frac{\epsilon}{3} \right\};\qquad B_2 ={\mathop \cap_{k=K}^{\lceil n^\delta\rceil}}\;\mathcal{M}(v_1, a\log_2 k, k)$$

By Lemma \ref{lem:SurCD} and the choice of $K_0$, when $n$ is large enough we have $\P_{p,n}(B_1 \cap B_2) > 1-\epsilon$. Assume $B_1\cap B_2$ occurs and fix $y^*$ with $\deg(y^*)>R^2$ and $d(v_1, y^*) < R$; let us now prove that, if $\frac{|\xi_0 \cap B(y^*,1)|}{\lambda\cdot |B(y^*,1)|} > \frac{1}{16e}$, then $P^\lambda_G(\xi^{y^*}_{t_n} \neq \varnothing) > 1-\epsilon$. Since $B_2$ occurs, we can take $z_K^*, z_{K+1}^*,\ldots, z^*_{\lceil n^\delta \rceil}$ such that $\deg(z^*_k) \geq k,\; d(v_1, z^*_k) \leq a\log_2 k$. Now, by (\ref{eq:pNGn1}) and (\ref{eq:pNGn2}),
$$\deg(y^*) > \frac{3}{c_1}\cdot \frac{1}{\lambda^2}\log\frac{1}{\lambda}  \cdot d(y^*, z_K^*) \text{ and } \deg(z^*_k) > \frac{3}{c_1}\cdot \frac{1}{\lambda^2}\log\frac{1}{\lambda} \cdot d(z^*_k, z^*_{k+1}) \text{ for all } k \geq K,$$
so Lemma \ref{lem:infNei} can be used repeatedly to guarantee that the infection is transmitted from $y^*$, through $z^*_K, z^*_{K+1},\ldots$ until $z^*_{\lceil n^\delta \rceil}$ with probability larger than
$$1-2e^{-c_1\lambda^2\deg(y^*)} -2\sum_{k=K}^{\lceil n^\delta \rceil } e^{-c_1\lambda^2\deg(z^*_k)} > 1-\frac{2}{3}\epsilon,$$ 
To conclude, by the definition of $B_1$, the infection then survives until time $t_n$ with probability larger than $1-\epsilon$.

\eprlem



\begin{thebibliography}{99}



\bibitem{BBCS} Berger, N., Borgs, C., Chayes, J.T., and Saberi, A. (2005) \textit{On the spread of viruses on the
internet}. Proceedings of the 16th Symposium on Discrete Algorithms, 301-310

\bibitem{CD} Chatterjee, S.  and Durrett, R. (2009)
\textit{Contact process on random graphs with degree power law distribution have critical value zero}.
Ann. Prob. 37, 2332-2356

\bibitem{dembzeit} Dembo, A. and Zeitouni, O. (1998) \textit{Large deviations techniques and applications}. 2nd edition. Springer, Application of Mathematics, vol. 38

\bibitem {remco2} Dommers, S., Giardin\`a, C., vd Hofstad, R. \textit{Ising critical exponents on random trees and graphs}, arXiv:1211.3005 (2012)

\bibitem{durprob} Durrett, R. (2010) \textit{Probability: Theory and Examples}. 4th edition. Cambridge Series in Statistical and Probabilistic Mathematics.

\bibitem{Dur} Durrett, R. (2007) \textit{Random Graph Dynamics}. Cambridge University Press

\bibitem{durliu} R.\ Durrett, X.\ Liu. \textit{The contact process on a finite set}. \emph{Ann. Probab.} \textbf{16}, 1158-1173 (1988).

\bibitem{remco} vd Hofstad, R. \textit{Random graphs and complex networks}. Available at http://www.win.tue.nl/{\textasciitilde{}}rhofstad/

\bibitem{lig85} T. Liggett, \emph{Interacting particle systems}. Grundlehren der mathematischen Wissenschaften \textbf{276}, Springer (1985).

\bibitem{lig99} T. Liggett, \emph{Stochastic interacting systems: contact, voter and exclusion processes}. Grund\-lehren der mathematischen Wissenschaften \textbf{324}, Springer (1999).

\bibitem{tommeta} T.\ Mountford. \textit{A metastable result for the finite multidimensional contact process}, \emph{Canad. Math. Bull.} \textbf{36} (2), 216-226 (1993).

\bibitem{MMVY} T. Mountford, J.C. Mourrat, D. Valesin, Q. Yao. \textit{Exponential extinction time of the contact process on finite graphs}, arXiv:1203.2972 (2012)

\bibitem{NSW} Newman, M.E.J., Strogatz, S.H., and Watts, D.J. (2001) \textit{Random graphs with arbitrary
degree distributions and their applications}. Phys. Rev. E. 64, paper 026118






\end{thebibliography}
\end{document}